\documentclass[11pt]{article}

\usepackage{amsmath, amsfonts, amssymb, amsthm,  graphicx, enumerate}
\usepackage[letterpaper, left=1truein, right=1truein, top = 1truein, bottom = 1truein]{geometry}

\usepackage[numbers, square]{natbib}

\usepackage[
            CJKbookmarks=true,
            bookmarksnumbered=true,
            bookmarksopen=true,
            colorlinks=true,
            citecolor=red,
            linkcolor=blue,
            anchorcolor=red,
            urlcolor=blue,
            ]{hyperref}

\usepackage[usenames,dvipsnames]{color}

\newcommand{\magenta}{\color{magenta}}

\usepackage[ruled, vlined, lined, commentsnumbered]{algorithm2e}

\usepackage{prettyref,soul,xspace}

\usepackage{tikz}
\usepackage{pgfplots,makecell}

\newcommand{\wh}{\widehat}
\newcommand{\wt}{\widetilde}
\newcommand{\bem}{\begin{bmatrix}}
\newcommand{\eem}{\end{bmatrix}}

\newcommand{\eps}{\epsilon}

\newcommand{\ie}{i.e.\xspace}
\newcommand{\iid}{i.i.d.\xspace}

\newcommand{\diff}{{\rm d}}

\newcommand{\Expect}{\mathbb{E}}

\newcommand{\Prob}{\mathbb{P}}

\newcommand{\iiddistr}{{\stackrel{\text{\iid}}{\sim}}}

\newcommand{\argmin}{\mathop{\rm argmin}}
\newcommand{\argmax}{\mathop{\rm argmax}}

\newcommand{\Th}{{^{\rm th}}}

\theoremstyle{remark}
\newtheorem{remark}{Remark}
\theoremstyle{plain}

\newtheorem{lemma}{Lemma}
\newtheorem{theorem}{Theorem}

\theoremstyle{definition}

\newtheorem{condition}{Condition}

\theoremstyle{plain}

\newtheorem{prop}{Proposition}

\newrefformat{eq}{(\ref{#1})}
\newrefformat{chap}{Chapter~\ref{#1}}
\newrefformat{sec}{Section~\ref{#1}}
\newrefformat{algo}{Algorithm~\ref{#1}}
\newrefformat{fig}{Fig.~\ref{#1}}
\newrefformat{tab}{Table~\ref{#1}}
\newrefformat{rmk}{Remark~\ref{#1}}
\newrefformat{clm}{Claim~\ref{#1}}
\newrefformat{def}{Definition~\ref{#1}}
\newrefformat{cor}{Corollary~\ref{#1}}
\newrefformat{lmm}{Lemma~\ref{#1}}
\newrefformat{lemma}{Lemma~\ref{#1}}
\newrefformat{prop}{Proposition~\ref{#1}}
\newrefformat{app}{Appendix~\ref{#1}}
\newrefformat{ex}{Example~\ref{#1}}
\newrefformat{exer}{Exercise~\ref{#1}}
\newrefformat{soln}{Solution~\ref{#1}}
\newrefformat{cond}{Condition~\ref{#1}}

\newcommand{\lunder}[1]{{\underset{\raise0.3em\hbox{$\smash{\scriptscriptstyle-}$}}{#1}}}

\newcommand{\norm}[1]{\left\|{#1} \right\|}
\newcommand{\Norm}[1]{\|{#1} \|}

\newcommand{\fnorm}[1]{\|#1\|_{\rm F}}

\newcommand{\opnorm}[1]{\|#1\|_{\rm op}}

\newcommand{\indc}[1]{{\mathbf{1}_{\left\{{#1}\right\}}}}
\newcommand{\Indc}{\mathbf{1}}

\def\innergetnumber#1[#2]#3{#2}
\def\getnumber{\expandafter\innergetnumber\jobname}

\newcommand{\calC}{{\mathcal{C}}}

\newcommand{\calE}{{\mathcal{E}}}

\newcommand{\pth}[1]{\left( #1 \right)}
\newcommand{\qth}[1]{\left[ #1 \right]}
\newcommand{\sth}[1]{\left\{ #1 \right\}}

\newtheorem{corollary}{Corollary}[section]

\newcommand{\nb}[1]{{\sf\magenta[#1]}}

\newcommand{\rev}[1]{{#1}}

\title{
Achieving Optimal Misclassification Proportion in Stochastic Block Model
}

\author{Chao Gao$^1$, Zongming Ma$^2$, Anderson Y.~Zhang$^1$ and Harrison H.~Zhou$^1$\\
~\\
$^1$Yale University and $^2$University of Pennsylvania
}

\begin{document}
\maketitle

\begin{abstract}
Community detection is a fundamental statistical problem in network data analysis. 
Many algorithms have been proposed 
to tackle this problem. 
Most of these algorithms are not guaranteed to achieve the statistical optimality of the problem, while procedures that achieve information theoretic limits for general parameter spaces are not computationally tractable. 
In this paper, we present a computationally feasible two-stage method that achieves optimal statistical performance in misclassification proportion for stochastic block model under weak regularity conditions. 
\rev{Our two-stage procedure consists of a refinement stage motivated by penalized local maximum likelihood estimation.
This stage can take a wide range of weakly consistent community detection procedures as initializer, to which it applies and outputs a community assignment that achieves optimal misclassification proportion with high probability.}
The practical effectiveness of the new algorithm is demonstrated by competitive numerical results.

\smallskip

\end{abstract}


\section{Introduction}
\label{sec:intro}

\newcommand{\loss}{\ell}

Network data analysis \citep{wasserman94,goldenberg10} has become one of the leading topics in statistics. 
In fields such as physics, computer science, social science and biology, one observes a network among a large number of subjects of interest such as particles, computers, people, etc.
The observed network can be modeled as an instance of a random graph and the goal is to infer structures of the underlying generating process. 
A structure of particular interest is \emph{community}: there is a partition of the graph nodes in some suitable sense so that each node belongs to a community. 
\rev{Starting with the proposal of a series of methodologies \citep{girvan02,newman07,handcock07,karrer11}, we have seen a tremendous literature devoted to algorithmic solutions to uncovering community structure and great advances have also been made in recent years on the theoretical understanding of the problem in terms of statistical consistency and thresholds for detection and exact recoveries. See, for instance, \cite{bickel09,decelle2011asymptotic,zhao12,mossel2012stochastic, mossel2013proof,massoulie2014community,abbe2014exact,mossel2014consistency,hajek2014achieving}, among others.
In spite of the great efforts exerted on this ``community detection" problem, its state-of-the-art solution has not yet reached the comparable level of maturity as what statisticians have achieved in other high dimensional problems such as nonparametric estimation \citep{tsybakov09,johnstone2011gaussian}, high dimensional regression \citep{bickel2009simultaneous} and covariance matrix estimation \citep{cai2010optimal}, etc.}
In these more well-established problems, not only do we know the fundamental statistical limits, we also have computationally feasible algorithms to achieve them. 
The major goal of the present paper is to serve as a step towards such maturity in network data analysis by proposing a computationally feasible algorithm for community detection in stochastic block model with provable statistical optimality.

To describe network data with community structure, we focus on the stochastic block model (SBM) proposed by \cite{holland83}.
Let $A\in \sth{0,1}^{n\times n}$ be the symmetric adjacency matrix of an undirected random graph generated according to an SBM with $k$ communities.
Then the diagonal entries of $A$ are all zeros and each $A_{uv} = A_{vu}$ for $u> v$ is an independent Bernoulli random variable with mean $P_{uv}=B_{\sigma(u)\sigma(v)}$ for some symmetric connectivity matrix $B\in [0,1]^{k\times k}$ and some label function $\sigma:[n]\rightarrow[k]$, where for any positive integer $m$, $[m] = \sth{1,\dots, m}$.
In other word, if the $u\Th$ node and the $v\Th$ node belong to the $i\Th$ and the $j\Th$ community respectively, then $\sigma(u)=i$, $\sigma(v)=j$ and there is an edge connecting $u$ and $v$ with probability $B_{ij}$. 
Community detection then refers to the problem of estimating the label function $\sigma$ subject to a permutation of the community labels $\sth{1,\dots, k}$. 
A natural loss function for such an estimation problem is the proportion of wrong labels (subject to a permutation of the label set $[k]$), which we shall refer to as misclassification proportion from here on.

In ground breaking works by 
Mossel et al.~\citep{mossel2012stochastic, mossel2013proof} and Massouli{\'e} \citep{massoulie2014community},  
the authors established sharp threshold for the regimes in which it is possible and impossible to achieve a misclassification proportion strictly less than $\frac{1}{2}$ when $k = 2$ and both communities are of the same size (so that it is better than random guess), which solved the conjecture in \cite{decelle2011asymptotic} that was only justified in physics rigor.
For some recent progress on the general case of fixed $k$ and possibly unequal sized communities, see \cite{abbe2015community}.
On the other hand, \citet{abbe2014exact,mossel2014consistency} and \citet{hajek2014achieving} established the necessary and sufficient condition for ensuring zero misclassification proportion (usually referred to as ``strong consistency'' in the literature) with high probability when $k=2$ and community sizes are equal, and was later generalized to a larger set of fixed $k$ by \cite{hajek2015achieving}.
Arguably, what is of more interest to statisticians is the intermediate regime between the above two cases, namely when the misclassification proportion is vanishing as the number of nodes grows but not exactly zero.
This is usually called the regime of ``weak consistency'' in the network literature.

To achieve weak (and strong) consistency, statisticians have proposed various methods. 
One popular approach is spectral clustering \citep{von2007tutorial} which is motivated by the observation that the rank of the $n\times n$ matrix $P = (P_{uv}) = (B_{\sigma(u)\sigma(v)})$ is at most $k$ and its leading eigenvectors contain information of the community structure.
The application of spectral clustering on network data goes back to \cite{hagen1992new,mcsherry2001spectral}, and its performance under the stochastic block model has been investigated by \cite{coja2010graph, rohe11,sussman2012consistent,fishkind2013consistent,qin2013regularized,joseph2013impact,lei14,vu2014simple,chin15,jin2015fast,le2015sparse}, among others.
To further improve the performance, various ways for refining spectral clustering have been proposed, such as those in \cite{amini13,mossel2014consistency,lei2014generic,yun2014accurate,chin15} which lead to strong consistency or convergence rates that are exponential in signal-to-noise ratio,
\rev{while \cite{mossel2013belief} studied the problem of minimizing a non-vanishing misclassification proportion.}
\rev{However, 
in the regime of weak consistency, these refinement methods are not guaranteed to attain the optimal misclassification proportion to be introduced below.}
Another important line of research is devoted to the investigation of likelihood-based methods, which was initiated by \cite{bickel09} and later extended to more general settings by \cite{zhao12,choi2012stochastic}.
\rev{To tackle the intractability of optimizing the likelihood function,}
an EM algorithm using pseudo-likelihood was proposed by \cite{amini13}.
Another way to overcome the intractability of the maximum likelihood estimator (MLE) is by convex relaxation.
Various semi-definite relaxations were studied by \cite{cai14, chen2014statistical, amini2014semidefinite}, and the aforementioned sharp threshold for strong consistency in \citep{hajek2014achieving, hajek2015achieving} was indeed achieved by semi-definite programming.
Recently, Zhang and Zhou \citep{yezhang15} established the minimax risk for misclassification proportion in SBM under weak conditions, which is of the form
\begin{equation}
\exp\left(-(1+o(1))\frac{nI^*}{k}\right)
\label{eq:minimax}
\end{equation}
if all $k$ communities are of equal sizes, where $I^*$ is the minimum R\'{e}nyi divergence of order $\frac{1}{2}$ \citep{rrnyi1961measures} of the within and the between community edge distributions.
See \prettyref{thm:minimax} below for a more general and precise statement of the minimax risk.
Unfortunately, \citet{yezhang15} used MLE for achieving the risk in \eqref{eq:minimax} which was hence computationally intractable.
Moreover, none of the spectral clustering based method or tractable variants of the likelihood based method has a known error bound that matches \eqref{eq:minimax} \rev{with the sharp constant $1+o(1)$ on the exponent}. 

The main contribution of the current paper lies in the proposal of a 
computationally feasible algorithm that provably achieves the optimal misclassification proportion established in \cite{yezhang15} adaptively under weak regularity conditions. 
It covers the cases of both finite and diverging number of communities and both equal and unequal community sizes and achieves both weak and strong consistency in the respective regimes. 
In addition, the algorithm is guaranteed to compute in polynomial time  even when the number of communities diverges with the number of nodes.
\rev{Since the error bound of the algorithm matches the optimal misclassification proportion in \cite{yezhang15} under weak conditions, it  achieves various existing detection boundaries in the literature.
For instance, for any fixed number of communities},
the procedure is weakly consistent under the necessary and sufficient condition of \cite{mossel2012stochastic, mossel2013proof},
and strongly consistent under the necessary and sufficient condition of \cite{abbe2014exact,mossel2014consistency,hajek2014achieving,hajek2015achieving}. 
Moreover, it could match the optimal misclassification proportion in \cite{yezhang15} even when $k$ diverges.
To the best of our limited knowledge, this is the first polynomial-time algorithm that achieves minimax optimal performance. 
In other words, the proposed procedure enjoys both statistical and computational efficiency.

The core of the algorithm is a refinement scheme for community detection \rev{motivated by penalized maximum likelihood estimation}. 
As long as there exists an initial estimator that satisfies a certain weak consistency criterion, the refinement scheme is able to obtain an improved estimator that achieves the optimal misclassification proportion in (\ref{eq:minimax}) with high probability. 
The key to achieve the bound in (\ref{eq:minimax}) is to optimize the \emph{local} penalized likelihood function for each node separately. 
This local optimization step is completely data-driven and has a closed form solution, and hence can be computed very efficiently. 
\rev{The additional penalty term is indispensable as it plays a key rule in ensuring the optimal performance when the community sizes are unequal and when the within community and/or between community edge probabilities are unequal.}

To obtain a qualified initial estimator, we show that both spectral clustering and its normalized variant could satisfy the desired condition needed for subsequent refinement, \rev{though the refinement scheme works for any other method satisfying a certain weak consistency condition}.
Note that spectral clustering can be considered as a \emph{global} method, and hence our two-stage algorithm runs in a ``\emph{from global to local}'' fashion.
In essence, with high probability,
the global stage pinpoints a local neighborhood in which we shall search for solution to each local penalized maximum likelihood problem, and the subsequent local stage finds the desired solution.
From this viewpoint, one can also regard our approach as an ``optimization after localization'' procedure.
Historically, this idea played a key role in the development of the renowned one-step efficient estimator \cite{barnett1966evaluation,le1969theorie,bickel75}.
It has also led to recent progress in non-convex optimization and localized gradient descent techniques for finding optimal solutions to high dimensional statistical problems.
Examples include but are not limited to high-dimensional linear regression \cite{zhang12general}, 
sparse PCA \citep{paul2012augmented,ma2013sparse,cai2013sparse,wang2014nonconvex}, sparse CCA \citep{gao2014sparse}, phase retrieval \cite{candes2014phase} and high dimensional EM algorithm \citep{wainwright2014statistical,wang2014high}.
A closely related idea has also found success in the development of confidence intervals for regression coefficients in high dimensional linear regression. See, for instance, \cite{zhang2014confidence,van2014asymptotically,javanmard2014confidence} and the references therein.
\rev{Last but not least, even when viewed as a ``spectral clustering plus refinement'' procedure, our method distinguishes itself from other such methods in the literature by provably achieving the minimax optimal performance over a wide range of parameter configurations.}

The rest of the paper is organized as follows. 
\prettyref{sec:method} formally sets up the community detection problem and presents the two-stage algorithm.
The theoretical guarantees for the proposed method are given in \prettyref{sec:result}, followed by numerical results demonstrating its competitive performance on both simulated and real datasets in Sections \ref{sec:simulation} and \ref{sec:realdata}.
A discussion on the results in the current paper and possible directions for future investigation is included in \prettyref{sec:discussion}.
\prettyref{sec:proof} presents the proofs of main results with some technical details deferred to the appendix.

We close this section by introducing some notation. For a matrix $M=(M_{ij})$, we denote its Frobenius norm by $\fnorm{M}=\sqrt{\sum_{ij}M_{ij}^2}$ and its operator norm by $\opnorm{M}=\max_l\lambda_l(M)$, where $\lambda_l(M)$ is its $l\Th$ singular value. We use $M_{i*}$ to denote its $i\Th$ row. The norm $\Norm{\cdot}$ is the usual Euclidean norm for vectors. 
For a set $S$, $|S|$ denotes its cardinality. 
The notation $\mathbb{P}$ and $\mathbb{E}$ are generic probability and expectation operators whose distribution is determined from the context. For two positive sequences $\{x_n\}$ and $\{y_n\}$, $x_n\asymp y_n$ means $x_n/C\leq y_n\leq Cx_n$ for some constant $C\geq 1$ independent of $n$. Throughout the paper, unless otherwise noticed, we use $C,c$ and their variants to denote absolute constants, whose values may change from line to line.

\section{Problem formulation and methodology}
\label{sec:method}

In this section, we give a precise formulation of the community detection problem and present a new method for it.
The method consists of two stages: initialization and refinement. We shall first introduce the second stage, which is the main algorithm of the paper. It clusters the network data by performing a node-wise penalized neighbor voting based on some initial community assignment.
Then, we will discuss several candidates for the initialization step including a new greedy algorithm for clustering the leading eigenvectors of the adjacency matrix or of the graph Laplacian that is tailored specifically for stochastic block model.
Theoretical guarantees for the algorithms introduced in the current section will be presented in \prettyref{sec:result}.

\subsection{Community detection in stochastic block model}

Recall that a stochastic block model is completely characterized by a symmetric connectivity matrix $B\in [0,1]^{k\times k}$ and a label vector $\sigma\in [k]^n$. 
One widely studied parameter space of SBM is 
\begin{align}
\nonumber
\Theta_0(n,k,a,b,\beta)
= \bigg\{&(B, \sigma):\, \sigma:[n]\to [k], |\sth{u\in [n]: \sigma(u) = i}|\in \left[\frac{n}{\beta k}-1, \frac{\beta n}{k}+1\right],~ \forall i\in [k], \\
\label{eq:para-space0} & B = (B_{ij})\in [0,1]^{k\times k}, B_{ii}=\frac{a}{n}\text{ for all }i\text{ and } B_{ij} = \frac{b}{n}\text{ for all }i\neq j \bigg\}
\end{align}
where $\beta \geq 1$ is an absolute constant.
This parameter space $\Theta_0(n,k,a,b,\beta)$ contains all SBMs 
in which the within community connection probabilities are all equal to $a\over n$ and the between community connection probabilities are all equal to $b\over n$. 
In the special case of $\beta = 1$, all communities are of nearly equal sizes.

Assuming equal within and equal between connection probabilities can be  restrictive. 
Thus, we also introduce the following larger parameter space
\begin{align}
\nonumber
\Theta(n,k,a,b,\lambda,\beta;\alpha)
= \bigg\{&(B, \sigma) :\,  \sigma:[n]\to [k], |\sth{u\in [n]: \sigma(u) = i}| \in\left[\frac{n}{\beta k}-1, \frac{\beta n}{k}+1\right],~ \forall i\in [k], \\
\nonumber & B =B^T= (B_{ij})\in [0,1]^{k\times k}, \frac{b}{\alpha n}\leq  \frac{1}{k(k-1)}\sum_{i\neq j}B_{ij} \leq\max_{i\neq j} B_{ij} = \frac{b}{n}, \\
\nonumber 
& \frac{a}{n} = \min_i B_{ii}\leq \max_i B_{ii} \leq \frac{\alpha a}{n},
\\ 
\label{eq:para-space} 
&  \lambda_k(P)\geq \lambda\text{ with }P=(P_{uv})=(B_{\sigma(u),\sigma(v)})\bigg\}.
\end{align}
Throughout the paper, we treat $\beta \geq 1$ and $\alpha\geq 1$ as absolute constants, while $k,a,b$ and $\lambda$ should be viewed as functions of the number of nodes $n$ which can vary as $n$ grows.
Moreover, we assume $0<\frac{b}{n}<\frac{a}{n}\leq 1-\epsilon$ throughout the paper for some numeric constant $\epsilon\in(0,1)$.
Thus, the parameter space $\Theta(n,k,a,b,\lambda,\beta;\alpha)$ requires that the within community connection probabilities are bounded from below by $a\over n$ and the connection probabilities between any two communities are bounded from above by $b\over n$. 
In addition, it requires that the sizes of different communities are comparable.
In order to guarantee that $\Theta(n,k,a,b,\lambda,\beta;\alpha)$ is a larger parameter space than $\Theta_0(n,k,a,b,\beta)$, we always require $\lambda$ to be positive and sufficiently small such that
\begin{equation}
\Theta_0(n,k,a,b,\beta)\subset \Theta(n,k,a,b,\lambda,\beta;\alpha).\label{eq:lambda}
\end{equation}
According to Proposition \ref{prop:eigen} in the appendix, a sufficient condition for 
(\ref{eq:lambda}) is $\lambda\leq \frac{a-b}{2\beta k}$. 
We assume (\ref{eq:lambda}) throughout the rest of the paper.

The labels on the $n$ nodes induce a community structure $[n]=\cup_{i=1}^k\mathcal{C}_i$, where $\mathcal{C}_i = \sth{u\in [n]: \sigma(u) = i}$ is the $i\Th$ community with size $n_i=|\mathcal{C}_i|$. 
Our goal is to reconstruct this partition, or equivalently, to estimate the label of each node modulo any permutation of label symbols. 
Therefore, a natural error measure is the misclassification proportion defined as
\begin{align}
\label{eq:loss}
\ell(\wh\sigma, \sigma) = \min_{\pi\in S_k}\frac{1}{n}\sum_{u\in [n]}\indc{\wh\sigma(u) \neq \pi(\sigma(u))},
\end{align}
where $S_k$ stands for the symmetric group on $[k]$ consisting of all permutations of $[k]$.

\subsection{Main algorithm}

We are now ready to present the main method of the paper -- 
\rev{a refinement algorithm for community detection in stochastic block model
motivated by penalized local maximum likelihood estimation}.

Indeed,
for any SBM in the parameter space $\Theta_0(n,k,a,b,1)$ with equal community size, the MLE for $\sigma$ \cite{cai14,chen2014statistical,yezhang15} is 
\begin{equation}
\wh{\sigma}=\argmax_{\sigma:[n]\rightarrow[k]}\sum_{u<v}A_{uv}\Indc_{\{\sigma(u)=\sigma(v)\}},
\label{eq:MLE}
\end{equation}
which is a combinatorial optimization problem and hence is computationally intractable. 
However, 
node-wise optimization of (\ref{eq:MLE}) has a simple closed form solution. 
Suppose the values of $\{\sigma(u)\}_{u=2}^n$ are known and we want to estimate $\sigma(1)$.
Then, (\ref{eq:MLE}) reduces to
\begin{equation}
\wh{\sigma}(1)=\argmax_{i\in[k]}\sum_{\{v\neq 1: \sigma(v)=i\}}A_{1v}.\label{eq:MLElocal}
\end{equation}
For each $i\in[k]$, the quantity $\sum_{\{v\neq 1: \sigma(v)=i\}}A_{1v}$ is exactly the number of neighbors that the first node has in the $i\Th$ community.
Therefore, the most likely label for the first node is the one it has the most connections with when all communities are of equal sizes.
In practice, we do not know any label in advance. However, we may estimate the labels of all but the first node by first applying a community detection algorithm $\sigma^0$ on the subnetwork excluding the first node and its associated edges, the adjacency matrix of which is denoted by $A_{-1}$ since it is the $(n-1)\times (n-1)$ submatrix of $A$ with its first row and first column removed.
Once we estimate the remaining labels, we can apply \eqref{eq:MLElocal} to estimate $\sigma(1)$ but with $\{\sigma(v)\}_{v=2}^n$ replaced with the estimated labels.

\begin{algorithm}[!htb]
	\SetAlgoLined
\caption{A refinement scheme for community detection 
\label{algo:refine}}
\KwIn{
Adjacency matrix $A\in \sth{0,1}^{n\times n}$,\\
~~~~~~~~~~~~number of communities $k$,\\
~~~~~~~~~~~~initial community detection method $\sigma^0$.
}

\KwOut{
Community assignment $\wh\sigma$.
}


\smallskip
{\bf Penalized neighbor voting:}\\
\nl \For{$u = 1$ \KwTo $n$}{
\nl Apply $\sigma^0$ on $A_{-u}$ to obtain $\sigma_u^0(v)$ for all $v\neq u$ and let $\sigma^0_u(u) = 0$\;
\nl Define $\wt{\mathcal{C}}_i^u = \sth{v: \sigma^0_u(v) = i}$ for all $i\in [k]$;
let $\wt\calE_i^u$ be the set of edges within $\wt{\mathcal{C}}_i^u$, and $\wt\calE_{ij}^u$ the set of edges between $\wt{\mathcal{C}}_i^u$ and $\wt{\mathcal{C}}_j^u$ when $i\neq j$\;
\nl Define
\begin{align}
	\label{eq:B-est}
\wh{B}_{ii}^u = \frac{|\wt\calE_i^u|}{\frac{1}{2}|\wt\calC_i^u|(|\wt\calC_i^u|-1) },
\quad 
\wh{B}_{ij}^u = 
\frac{|\wt\calE_{ij}^u|}{|\wt\calC_i^u||\wt\calC_j^u|},
\quad
\forall i\neq j\in [k],
\end{align}
and let 
\begin{align}
	\label{eq:ab-esti}
\wh{a}_u = n\min_{i\in [k]} \wh{B}_{ii}^u 
\quad \mbox{and} \quad
\wh{b}_u = n\max_{i\neq j\in [k]} \wh{B}_{ij}^u.
\end{align}

\nl Define $\wh\sigma_u: [n]\to [k]$ by setting $\wh\sigma_u(v)=\sigma^0_u(v)$ for all $v\neq u$ and 
\begin{align}
\label{eq:refine-nb}
\wh\sigma_u(u) = \argmax_{l\in [k]} \sum_{\sigma^0_u(v) = l} A_{uv} - \rho_u \sum_{v\in [n]} \indc{\sigma^0_u(v)= l}
\end{align}
where for
\begin{align}
\label{eq:t-set}
t_u = \frac{1}{2}\log
{\frac{\wh{a}_u(1 - \wh{b}_u/n)}{\wh{b}_u(1 - \wh{a}_u/n)}},
\end{align}
we define
\begin{equation}
	\label{eq:lambda-set}
\rho_u = 
-\frac{1}{2t_u}\log\pth{\frac{\frac{\wh{a}_u}{n}e^{-t_u}+1 -\frac{\wh{a}_u}{n}}{\frac{\wh{b}_u}{n}e^{t_u}+1 -\frac{\wh{b}_u}{n}}},
\end{equation}

}

%

{\bf Consensus:}\\

\nl 
Define $\wh\sigma(1) = \wh\sigma_1(1)$.
For $u=2,\dots, n$, 
define 
\begin{align}
\label{eq:consensus}
\wh\sigma(u) = \argmax_{l\in [k]} 
\left|\{v:\wh\sigma_1(v) = l\} \cap \{v: \wh\sigma_u(v) = \wh\sigma_u(u)\} \right|.
\end{align}
\end{algorithm}

For any $u\in [n]$, let 
$A_{-u}$ denote the $(n-1)\times (n-1)$ submatrix of $A$ with its $u\Th$ row and $u\Th$ column removed.
Given any community detection algorithm $\sigma^0$ which is able to cluster any graph on $n-1$ nodes into $k$ categories, 
we present the precise description of our refinement scheme in \prettyref{algo:refine}.

The algorithm works in two consecutive steps.
The first step carries out the foregoing heuristics on a node by node basis.
For each fixed node $u$, we first leave the node out and apply the available community detection algorithm $\sigma^0$ on the remaining $n-1$ nodes and the edges among them (as summarized in the matrix $A_{-u}\in \sth{0,1}^{(n-1)\times (n-1)}$)
to obtain an initial community assignment vector $\sigma^0_u$.
For convenience, we make $\sigma^0_u$ an $n$-vector by fixing $\sigma^0_u(u) = 0$, though applying $\sigma^0$ on $A_{-u}$ does not give any community assignment for $u$.
We then assign the label of the $u\Th$ node according to \eqref{eq:refine-nb}, which is essentially \eqref{eq:MLElocal} with $\sigma$ replaced with $\sigma^0_u$ except for the additional penalty term.
The additional penalty term is added to ensure the optimal performance even when both the diagonal and the off-diagonal entries of the connectivity matrix $B$ are allowed to take different values and the community sizes are not necessarily equal.
To determine the penalty parameter $\rho_u$ in an adaptive way as spelled out in \eqref{eq:t-set} -- \eqref{eq:lambda-set}, we first estimate the connectivity matrix $B$ based on $A_{-u}$ in \eqref{eq:B-est} -- \eqref{eq:ab-esti}.
After we obtain the community assignment for $u$, we organize the assignment for all $n$ vertices into an $n$-vector $\wh\sigma_u$.
\rev{We call this step ``penalized neighbor voting'' since the first term on the RHS of \eqref{eq:refine-nb} counts the number of neighbors of $u$ in each (estimated) community while the second term is a penalty term proportional to the size of each (estimated) community}.

Once we complete the above procedure for each of the $n$ nodes, we obtain $n$ vectors $\wh\sigma_u\in [k]^n$, $u=1,\dots, n$, and turn to the second step of the algorithm.
The basic idea behind the second step is to obtain a unified community assignment by assembling $\{\wh\sigma_u(u):u\in [n] \}$ and the immediate hurdle is that each $\wh\sigma_u$ is only determined up to a permutation of the community labels. 
Thus, the second step aims to find the right permutations by \eqref{eq:consensus} before we assemble the $\wh\sigma_u(u)$'s.
We call this step ``consensus'' since we are essentially looking for a consensus on the community labels for $n$ possibly different community assignments, under the assumption that all of them are close to the ground truth up to some permutation.

%



%
%

\subsection{Initialization via spectral methods}
\label{sec:initial}

In this section,
we present algorithms that can be used as initializers in \prettyref{algo:refine}. 
Note that for any model in \eqref{eq:para-space}, the matrix $P$ has rank at most $k$ and $\Expect A_{uv} = P_{uv}$ for all $u\neq v$. 
We may first reduce the dimension of the data and then apply some clustering algorithm. 
Such an approach is usually referred to as spectral clustering \citep{von2007tutorial}. 
Technically speaking, 
spectral clustering refers to the general method of clustering eigenvectors of some data matrix. For random graphs, two commonly used methods are called unnormalized spectral clustering (USC) and normalized spectral clustering (NSC). The former refers to clustering the eigenvectors of the adjacency matrix $A$ itself and the latter refers to clustering the eigenvectors of the associated graph Laplacian $L(A)$. To formally define the graph Laplacian, we introduce the notation  $d_u=\sum_{v\in[n]}A_{uv}$ for the degree of the $u\Th$ node. 
The graph Laplacian operator $L:A\mapsto L(A)$ is defined by $L(A) = ([L(A)]_{uv})$ where $[L(A)]_{uv}=d_u^{-1/2}d_v^{-1/2}A_{uv}$. 
Although there have been debates and studies on which one works better (see, for example, \cite{von2008consistency, sarkar2013role}), 
for our purpose, 
both of them can lead to sufficiently decent initial estimators.

The performances of USC and NSC depend critically on the bounds $\opnorm{A-P}$ and $\opnorm{L(A)-L(P)}$, respectively. However, as pointed out by \cite{chin15,le2015sparse}, the matrices $A$ and $L(A)$ are not good estimators of $P$ and $L(P)$ under the operator norm when the graph is sparse in the sense that $\max_{u,v\in[n]}P_{uv}=o(\log n/n)$. Thus, regularizing $A$ and $L(A)$ are necessary to achieve better performances for USC and NSC. The adjacency matrix $A$ can be regularized by trimming those nodes with high degrees.
Define the trimming operator $T_{\tau}: A\mapsto T_{\tau}(A)$ by replacing the $u\Th$ row and the $u\Th$ column of $A$ with $0$ whenever $d_u\geq\tau$, and so $T_\tau(A)$ and $A$ are of the same dimensions.
It is argued in \cite{chin15} that by removing those high-degree nodes, $T_{\tau}(A)$ has better convergence properties. 
Regularization method for graph Laplacian goes back to \cite{amini13} and its theoretical properties have been studied by \cite{joseph2013impact, le2015sparse}. 
In particular, \citet{amini13} proposed to use $L(A_{\tau})$ for NSC where $A_{\tau}=A+\frac{\tau}{n}\mathbf{1}\mathbf{1}^T$ and $\mathbf{1}=(1,1,...,1)^T\in\mathbb{R}^n$. From now on, we use USC$(\tau)$ and NSC$(\tau)$ to denote unnormalized spectral clustering and normalized spectral clustering with regularization parameter $\tau$, respectively.
Note that the unregularized USC is USC$(\infty)$ and the unregularized NSC is NSC$(0)$.

Another important issue in spectral clustering lies in the subsequent clustering method used to cluster the  eigenvectors.
A popular choice is $k$-means clustering. However, finding the global solution to the $k$-means problem is NP-hard \citep{aloise2009np,mahajan2009planar}.
\citet{kumar2004simple} proposed a polynomial time algorithm for achieving $(1+\epsilon)$ approximation to the $k$-means problem for any fixed $k$, which was utilized in \cite{lei14} to establish consistency for spectral clustering under stochastic block model with fixed number of communities. 
However, a closer look at the complexity bound suggests that the smallest possible $\epsilon$ is proportional to $k$.
Thus, applying the algorithm and the associated bound in \cite{kumar2004simple} directly in our settings can lead to inferior error bounds when $k\to\infty$ as $n\to \infty$.
To address this issue under stochastic block model, we propose a greedy clustering algorithm in \prettyref{algo:greedy} inspired by the fact that the clustering centers under stochastic block model are well separated from each other on the population level.
It is straightforward to check that the complexity of Algorithm \ref{algo:greedy} is polynomial in $n$.

\begin{algorithm}[!bth]
	\SetAlgoLined
\caption{A greedy method for clustering 
\label{algo:greedy}}
\KwIn{
Data matrix $\wh{U}\in \mathbb{R}^{n\times k}$, either the leading eigenvectors of $T_{\tau}(A)$ or that of $L(A_{\tau})$,\\
~~~~~~~~~~~~number of communities $k$,\\
~~~~~~~~~~~~critical radius $r=\mu\sqrt{\frac{k}{n}}$ with some constant $\mu>0$.
}

\KwOut{
Community assignment $\wh\sigma$.
}

\smallskip
\nl Set $S=[n]$\;

\smallskip
\nl \For{$i = 1$ \KwTo $k$}{
\nl Let $t_i=\arg\max_{u\in S}\left|\left\{v\in S:\norm{\wh{U}_{v*}-\wh{U}_{u*}}<r\right\}\right|$\;
\nl Set $\wh{\mathcal{C}}_i=\left\{v\in S: \norm{\wh{U}_{v*}-\wh{U}_{t_i*}}<r\right\}$\;
\nl Label $\wh{\sigma}(u)=i$ for all $u\in\wh{\mathcal{C}}_i$\;
\nl Update $S\leftarrow S\backslash\wh{\mathcal{C}}_i$.
}
\smallskip
\nl If $S\neq\varnothing$, then for any $u\in S$, set $\wh{\sigma}(u)=\argmin_{i\in[k]}\frac{1}{|\wh{\mathcal{C}}_i|}\sum_{v\in\wh{\mathcal{C}}_i}\norm{\wh{U}_{u*}-\wh{U}_{v*}}$.
\end{algorithm}

Last but not least, we would like to emphasize that one needs not limit the initialization algorithm to the spectral methods introduced in this section.
As \prettyref{thm:upper} below shows, \prettyref{algo:refine} works for any initialization method that satisfies a weak consistency condition.


\section{Theoretical properties}
\label{sec:result}

Before stating the theoretical properties of the proposed method, we first review the minimax rate in \cite{yezhang15},
which will be used as the optimality benchmark.
The minimax risk is governed by the following critical quantity,
\begin{equation}
	\label{eq:I}
I^*=-2\log\left(\sqrt{\frac{a}{n}}\sqrt{\frac{b}{n}}+\sqrt{1-\frac{a}{n}}\sqrt{1-\frac{b}{n}}\right),
\end{equation}
which is the R\'{e}nyi divergence of order $\frac{1}{2}$ between $\mbox{Bern}\pth{\frac{a}{n}}$ and $\mbox{Bern}\pth{\frac{b}{n}}$, \ie,  Bernoulli distributions with success probabilities $\frac{a}{n}$ and $\frac{b}{n}$ respectively. Recall that $0<\frac{b}{n}<\frac{a}{n}\leq 1-\epsilon$ is assumed throughout the paper.
It can be shown that $I^*\asymp\frac{(a-b)^2}{na}$. 
Moreover, when $\frac{a}{n}=o(1)$,
\begin{align*}
I^* & = (1+o(1)) \frac{(\sqrt{a}-\sqrt{b})^2}{n}
 =(1+o(1)) \left[\left(\sqrt{\frac{a}{n}}-\sqrt{\frac{b}{n}}\right)^2+\left(\sqrt{1-\frac{a}{n}}-\sqrt{1-\frac{b}{n}}\right)^2 \right] 
\\
& 
= (2+o(1))  H^2\left(\mbox{Bern}\pth{\tfrac{a}{n}}, \mbox{Bern}\pth{\tfrac{b}{n}} \right),
\end{align*}
where $H^2(P,Q) = \frac{1}{2}\int(\sqrt{\diff P} - \sqrt{\diff Q})^2$ is the squared Hellinger distance between two distributions $P$ and $Q$.
The minimax rate for the parameter spaces (\ref{eq:para-space0}) and (\ref{eq:para-space}) under the loss function (\ref{eq:loss}) is given in the following theorem.
\begin{theorem}[\cite{yezhang15}]\label{thm:minimax}
When
$\frac{(a-b)^2}{ak\log k}\rightarrow \infty$, we have
$$\inf_{\wh{\sigma}}\sup_{(B,\sigma)\in\Theta}\mathbb{E}_{B,\sigma}\ell(\wh\sigma,\sigma)=\begin{cases} \exp\left(-(1+\eta)\frac{nI^*}{2}\right), & k=2; \\
\exp\left(-(1+\eta)\frac{nI^*}{\beta k}\right), & k\geq 3,
\end{cases}
$$
for both $\Theta=\Theta_0(n,k,a,b,\beta)$ and $\Theta=\Theta(n,k,a,b,\lambda,\beta;\alpha)$ with any $\lambda\leq \frac{a-b}{2\beta k}$
and any $\beta\in [1,\sqrt{5/3})$, where $\eta=\eta_n\rightarrow 0$ is some sequence tending to $0$ as $n\rightarrow\infty$. 
\end{theorem}

\begin{remark}
The assumption $\beta\in [1,\sqrt{5/3})$ is needed in \cite{yezhang15} for some technical reason. 
\rev{Here, the parameter $\beta$ enters the minimax rates when $k\geq 3$ since the worst case is essentially when one has two communities of size $\frac{n}{\beta k}$, while for $k=2$, the worst case is essentially two communities of size $\frac{n}{2}$.}
For all other results in this paper, we allow $\beta$ to be an arbitrary constant no less than $1$.
\end{remark}

To this end, let us show that the two-stage algorithm proposed in Section \ref{sec:method} achieves the optimal misclassification proportion.
The essence of the two-stage algorithm lies in the refinement scheme described in Algorithm \ref{algo:refine}. As long as any initialization step satisfies a certain weak consistency criterion, the refinement step directly leads to a solution with optimal misclassification proportion.
To be specific, the initialization step needs to satisfy the following condition.

\begin{condition}
\label{cond:init}
There exist constants $C_0, \delta > 0$ and a positive sequence $\gamma = \gamma_n$ such that
\begin{align} 
\inf_{(B,\sigma)\in \Theta}
\min_{u\in [n]}
\Prob_{B,\sigma}\sth{\loss(\sigma,\sigma^0_u)\leq \gamma} \geq 1 - C_0 n^{-(1+\delta)},
\end{align}
for some parameter space $\Theta$.
\end{condition}

Under \prettyref{cond:init}, we have the following upper bounds regarding the performance of the proposed refinement scheme.
\begin{theorem}
	\label{thm:upper}
Suppose as $n\to\infty$,
$\frac{(a-b)^2}{ak\log k} \to \infty$, $a\asymp b$ and \prettyref{cond:init} is satisfied for
\begin{align}
\label{eq:gamma-cond-1}
\gamma = o\pth{\frac{1}{k\log k}}
\end{align}
and $\Theta = \Theta_0(n,k,a,b,\beta)$.
Then
there is a sequence $\eta\to 0$ such that
\begin{equation}
	\label{eq:upper-bound}
\begin{aligned}
& \sup_{(B,\sigma)\in \Theta}\Prob_{B,\sigma}\sth{
\loss(\sigma, \wh{\sigma}) \geq
\exp\pth{-(1-\eta) \frac{n I^*}{2}}
} \to 0,  & \quad \mbox{if $k = 2$},\\
& \sup_{(B,\sigma)\in \Theta}\Prob_{B,\sigma}\sth{
\loss(\sigma, \wh{\sigma}) \geq
\exp\pth{-(1-\eta) \frac{n I^*}{\beta k}}
} \to 0, & \quad \mbox{if $k \geq 3$},
\end{aligned}
\end{equation}
where $I^*$ is defined as in \eqref{eq:I}.

If in addition \prettyref{cond:init} is satisfied for $\gamma$ satisfying both \eqref{eq:gamma-cond-1} and
\begin{align}
	\label{eq:gamma-cond-2}
\gamma  = o\pth{\frac{a-b}{ak}}
\end{align}
and $\Theta = \Theta(n,k,a,b,\lambda,\beta;\alpha)$, then the conclusion in \eqref{eq:upper-bound} continues to hold for $\Theta = \Theta(n,k,a,b,\lambda,\beta;\alpha)$.
\end{theorem}

Theorem \ref{thm:upper} assumes $a\asymp b$. The case when $a\asymp b$ may not hold is considered in Section \ref{sec:discussion}.
Compared with Theorem \ref{thm:minimax}, the upper bounds (\ref{eq:upper-bound}) achieved by Algorithm \ref{algo:refine} is minimax optimal. The condition (\ref{eq:gamma-cond-1}) for the parameter space $ \Theta_0(n,k,a,b,\beta)$ is very mild. When $k=O(1)$, it reduces to $\gamma=o(1)$ and simply means that the initialization should be weakly consistent at any rate. For $k\rightarrow\infty$, it implies that the misclassification proportion within each community converges to zero. Note that if the initialization step gives  wrong labels to all nodes in one particular community, then the misclassification proportion is at least $1/k$. The condition (\ref{eq:gamma-cond-1}) rules out this situation.  For the parameter space $\Theta(n,k,a,b,\lambda,\beta;\alpha)$, an extra condition (\ref{eq:gamma-cond-2}) is required. This is because estimating the connectivity matrix $B$ in $\Theta(n,k,a,b,\lambda,\beta;\alpha)$ is harder than in $\Theta_0(n,k,a,b,\beta)$. In other words, if we do not pursue adaptive estimation, (\ref{eq:gamma-cond-2}) is not needed.

\begin{remark}
\label{rmk:weak-gamma}
Theorem \ref{thm:upper} is an adaptive result without assuming the knowledge of $a$ and $b$. When these two parameters are known, we can directly use $a$ and $b$ in (\ref{eq:t-set}) of Algorithm \ref{algo:refine}. By scrutinizing the proof of Theorem \ref{thm:upper}, the conditions (\ref{eq:gamma-cond-1}) and (\ref{eq:gamma-cond-2}) can be weakened as $\gamma=o(k^{-1})$ in this case.
\end{remark}

Given the results of Theorem \ref{thm:upper}, it remains to check the initialization step via spectral clustering satisfies Condition \ref{cond:init}. 
For matrix $P=(P_{uv})=(B_{\sigma(u)\sigma(v)})$ with $(B,\sigma)$ belonging to either $\Theta_0(n,k,a,b,\beta)$ or $\Theta(n,k,a,b,\lambda,\beta;\alpha)$, we use $\lambda_k$ to denote $\lambda_k(P)$. Define the average degree by
\begin{equation}
\bar{d}=\frac{1}{n}\sum_{u\in[n]}d_u.\label{eq:ave-deg}
\end{equation}

\begin{theorem}\label{thm:ini1}
Assume $e\leq a\leq  C_1b$ for some constant $C_1>0$ and
\begin{equation}
\frac{ka}{\lambda_k^2}\leq c,\label{eq:assume}
\end{equation}
for some sufficiently small $c\in (0,1)$. Consider USC$(\tau)$ with a sufficiently small constant $\mu>0$ in Algorithm \ref{algo:greedy} and $\tau=C_2\bar{d}$ for some sufficiently large constant $C_2>0$. For any constant $C'>0$, there exists some $C>0$ only depending on $C',C_1,C_2$ and $\mu$ such that
$$\ell(\wh{\sigma},\sigma)\leq C\frac{a}{\lambda_k^2},$$
with probability at least $1-n^{-C'}$. If $k$ is fixed, the same conclusion holds  without assuming $a\leq C_1b$.
\end{theorem}

\begin{remark}
Theorem \ref{thm:ini1} improves the error bound for spectral clustering in \cite{lei14}. While \cite{lei14} requires the assumption $a>C\log n$, our result also holds for $a=o(\log n)$. A result close to ours is that by \cite{chin15}, but their clustering step is different from Algorithm \ref{algo:greedy}. Moreover, the conclusion of Theorem \ref{thm:ini1} holds with probability $1-n^{-C'}$ for an arbitrary large $C'$, which is critical because the initialization step needs to satisfy Condition \ref{cond:init} for the subsequent refinement step to work. On the other hand, the bound in \cite{chin15} is stated with probability $1-o(1)$.
\end{remark}

When $k=O(1)$, Theorem \ref{thm:upper} and Theorem \ref{thm:ini1} jointly imply the following result.
\begin{corollary}
Consider Algorithm \ref{algo:refine} initialized by $\sigma^0$ with USC$(\tau)$ for $\tau=C\bar{d}$, where $C$ is a sufficiently large constant. Suppose as $n\rightarrow\infty$, $k=O(1)$, $\frac{(a-b)^2}{a}\rightarrow\infty$ and $a\asymp b$. Then, there exists a sequence $\eta\rightarrow 0$ such that
\begin{equation*}
\begin{aligned}
& \sup_{(B,\sigma)\in \Theta}\Prob_{B,\sigma}\sth{
\loss(\sigma, \wh{\sigma}) \geq
\exp\pth{-(1-\eta) \frac{n I^*}{2}}
} \to 0,  & \quad \mbox{if $k = 2$},\\
& \sup_{(B,\sigma)\in \Theta}\Prob_{B,\sigma}\sth{
\loss(\sigma, \wh{\sigma}) \geq
\exp\pth{-(1-\eta) \frac{n I^*}{\beta k}}
} \to 0, & \quad \mbox{if $k \geq 3$},
\end{aligned}
\end{equation*}
where the parameter space is $\Theta = \Theta_0(n,k,a,b,\beta)$.
\end{corollary}
Compared with Theorem \ref{thm:minimax}, the proposed procedure achieves the minimax rate under the condition $\frac{(a-b)^2}{a}\rightarrow\infty$ and $a\asymp b$. When $k=O(1)$, the condition $\frac{(a-b)^2}{a}\rightarrow\infty$ is necessary and sufficient for weak consistency in view of Theorem \ref{thm:minimax}. More general results including the case of $k\rightarrow\infty$ are stated and discussed in Section \ref{sec:discussion}.

The following theorem characterizes the misclassification rate of normalized spectral clustering. 

\begin{theorem}\label{thm:ini2}
Assume $e\leq a\leq  C_1b$ for some constant $C_1>0$ and
\begin{equation}
\frac{ka\log a}{\lambda_k^2} \leq c,\label{eq:assume2}
\end{equation}
for some sufficiently small $c\in(0,1)$. Consider NSC$(\tau)$ with a sufficiently small constant $\mu>0$ in Algorithm \ref{algo:greedy} and $\tau=C_2\bar{d}$ for some sufficiently large constant $C_2>0$. Then, for any constant $C'>0$, there exists some $C>0$ only depending on $C',C_1,C_2$ and $\mu$ such that
$$\ell(\wh{\sigma},\sigma)\leq C\frac{a\log a}{\lambda_k^2},$$
with probability at least $1-n^{-C'}$. If $k$ is fixed, the same conclusion holds  without assuming $a\leq C_1b$.
\end{theorem}

\begin{remark}
A slightly different regularization of normalized spectral clustering is studied by \cite{qin2013regularized} only for the dense regime, while Theorem \ref{thm:ini2} holds under both dense and sparse regimes. Moreover, our result also improves that of \cite{le2015sparse} due to our tighter bound on $\opnorm{L(A_{\tau})-L(P_{\tau})}$ in Lemma \ref{lem:sp2} below. We conjecture that the $\log a$ factor in both the assumption and the bound of Theorem \ref{thm:ini2} can be removed.
\end{remark}

Note that Theorem \ref{thm:ini1} and Theorem \ref{thm:ini2} are stated in terms of the quantity  $\lambda_k$. We may specialize the results into the parameter spaces defined in (\ref{eq:para-space0}) and (\ref{eq:para-space}). By Proposition \ref{prop:eigen}, $\lambda_k\geq \frac{a-b}{2\beta k}$ for $\Theta_0(n,k,a,b,\beta)$ and $\lambda_k\geq \lambda$ for $\Theta(n,k,a,b,\lambda,\beta;\alpha)$. The implications of Theorem \ref{thm:ini1} and Theorem \ref{thm:ini2} and their uses as the initialization step for Algorithm \ref{algo:refine} are discussed in full details in Section \ref{sec:discussion}.





\section{Numerical results}
\label{sec:simulation}
In this section we present the performance of the proposed algorithm on simulated datasets.
The experiments cover three different scenarios: (1) dense network with communities of equal sizes; (2) dense network with communities of unequal sizes; and (3) sparse network. 
Recall the definition of $\bar d$ in (\ref{eq:ave-deg}). 
For each setting, we report results of Algorithm \ref{algo:refine} initialized with four different approaches: USC$(\infty)$, USC$(2\bar{d})$, NSC$(0)$ and NSC$(\bar{d})$,
the description of which can all be found in \prettyref{sec:initial}.
For all these spectral clustering methods, \prettyref{algo:greedy} was used to cluster the leading eigenvectors. 
The constant $\mu$ in the critical radius definition was set to be $0.5$ 
in all the results reported here.
For each setting, the results are based on 100 independent draws from the underlying stochastic block model.

To achieve faster running time, we also ran a simplified version of Algorithm \ref{algo:refine}. 
Instead of obtaining $n$ different initializers $\{\sigma_u\}_{u\in[n]}$ to refine each node separately, the simplified algorithm refines all the nodes with a single initialization on the whole network. 
Thus, 
the running time can be reduced roughly by a factor of $n$. 
Simulation results below suggest that the simplified version achieves similar performances to that of \prettyref{algo:refine} in all the settings we have considered.
For the precise description of the simplified algorithm, we refer readers to \prettyref{algo:simple} in the appendix.

\paragraph{Balanced case}
In this setting, we generate networks with 2500 nodes and 10 communities, each of which consists of 250 nodes, and we set $B_{ii} = 0.48$ for all $i$ and $B_{ij} = 0.32$ for all $i\neq j$.
Figure \ref{fig:ba} shows the boxplots of the number of misclassified nodes.
The first four boxplots correspond to the four different spectral clustering methods, in the order of USC$(\infty)$, USC$(2\bar{d})$, NSC$(0)$ and NSC$(\bar{d})$.
The middle four correspond to the results achieved by applying the simplified refinement scheme with these four initialization methods, and the last four show the results of \prettyref{algo:refine} with these four initialization methods.
Regardless of the initialization method, \prettyref{algo:refine} or its simplified version reduces the number of misclassified nodes from around 30 to around 5.

\begin{figure}[ht]
\centering
\includegraphics[trim=0mm 8mm 0mm 2mm, clip, width=\textwidth]{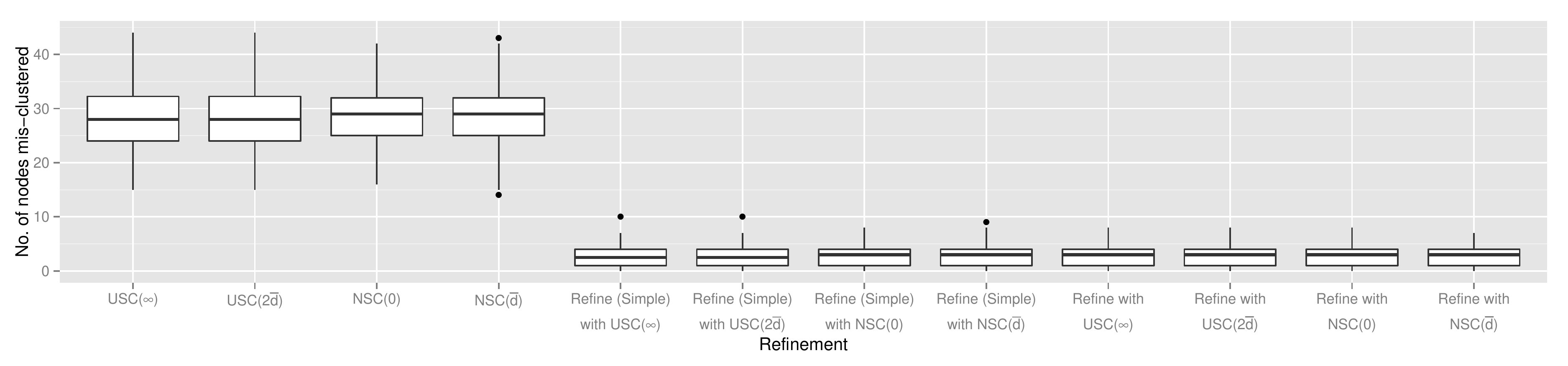}

\caption{Boxplots of number of misclassified nodes: Balanced case. 
\textit{Simple} indicates that the simplified version of Algorithm \ref{algo:refine} is used instead.\label{fig:ba}}
\end{figure}

\paragraph{Imbalanced case}
In this setting, we generate networks with 2000 nodes and 4 communities, the sizes of which are $200,400,600$ and $800$, respectively. The connectivity matrix is
\begin{align*}
B=
\begin{pmatrix}
  0.50 & 0.29 & 0.35 & 0.25 \\
  0.29 & 0.45 & 0.25 & 0.30 \\
  0.35 & 0.25 & 0.50 & 0.35  \\
  0.25 & 0.30 & 0.35 & 0.45
 \end{pmatrix}.
\end{align*}
Hence, the within-community edge probability is no smaller than 0.45 while the between-community edge probability is no greater than 0.35, and
the underlying SBM is inhomogeneous.
Figure \ref{fig:in} shows the boxplots of the number of misclassified nodes obtained by different initialization methods and their refinements, and the boxplots are presented in the same order as those in Figure \ref{fig:ba}.
Similarly, we can see refinement significantly reduces the error. 
\begin{figure}[ht]
\centering
\includegraphics[trim=0mm 8mm 0mm 2mm, clip, width=\textwidth]{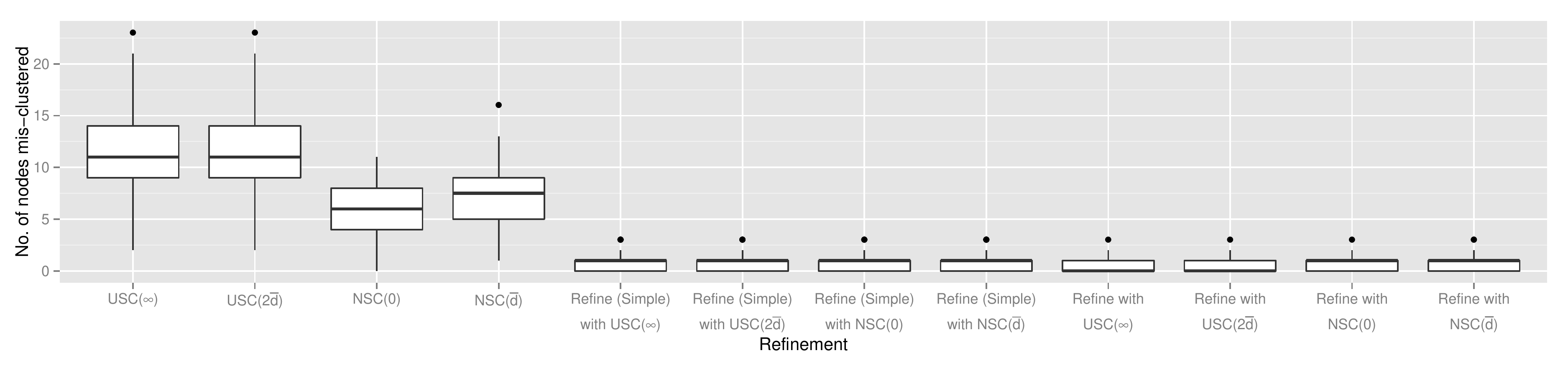}
\caption{Boxplots of number of misclassified nodes: imbalanced case. 
\textit{Simple} indicates that the simplified version of Algorithm \ref{algo:refine} is used instead.\label{fig:in}}
\end{figure}

\paragraph{Sparse case} 
In this setting we consider a much sparser stochastic block model than the previous two cases. 
In particular, each simulated network has 4000 nodes, divided into 10 communities all of size 400. We set all $B_{ii}=0.032$ and all $B_{ij} = 0.005$ when $i\neq j$. 
The average degree of each node in the network is around 30. 
Figure \ref{fig:sp} shows the boxplots of the number of misclassified nodes obtained by different initialization methods and their refinements, and the boxplots are presented in the same order as those in Figure \ref{fig:ba}.
Compared with either USC or NSC initialization, refinement reduces the number of misclassified nodes by 50\%.

\begin{figure}[ht]
\centering
\includegraphics[trim=0mm 8mm 0mm 2mm, clip, width=\textwidth]{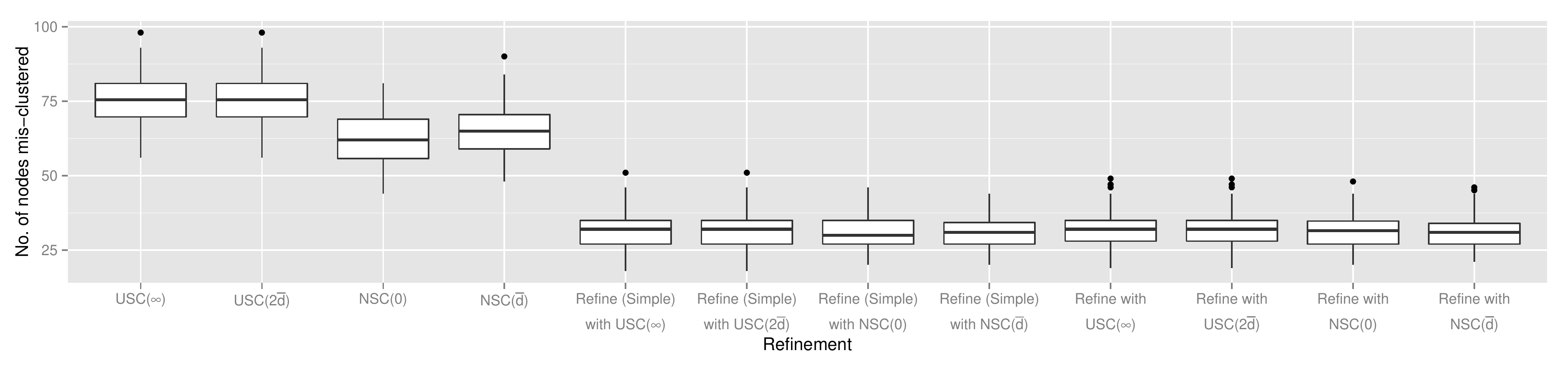}
\caption{Boxplots of number of misclassified nodes: Sparse case. 
\textit{Simple} indicates that the simplified version of Algorithm \ref{algo:refine} is used instead.\label{fig:sp}}
\end{figure}

\paragraph{Summary} In all three simulation settings, for all four initialization approaches considered, the refinement scheme in \prettyref{algo:refine} (and its simplified version) was able to significantly reduce the number of misclassified nodes, which is in agreement with the theoretical properties presented in \prettyref{sec:result}.


\section{Real data example}
\label{sec:realdata}
We now compare the results of our algorithm and some existing methods on a political blog dataset \citep{adamic2005political}. 
Each node in this network represents a blog about US politics and a pair of nodes is connected if one blog contains a link to the other.
There were 1490 nodes to start with, each labeled liberal or conservative. 
In what follows, we consider only the 1222 nodes located in the largest connected component of the network. 
This pre-processing step is the same as what was done in \cite{karrer11}. 
After pre-processing, 
the network has 586 liberal blogs and 636 conservative ones which naturally form two communities. 
As shown in the right panel of Figure \ref{fig:connectivity}, nodes are more likely to be connected if they have the same political ideology.

\begin{figure}
\centering
\includegraphics[trim=20mm 20mm 20mm 20mm, clip, width=0.4\textwidth]{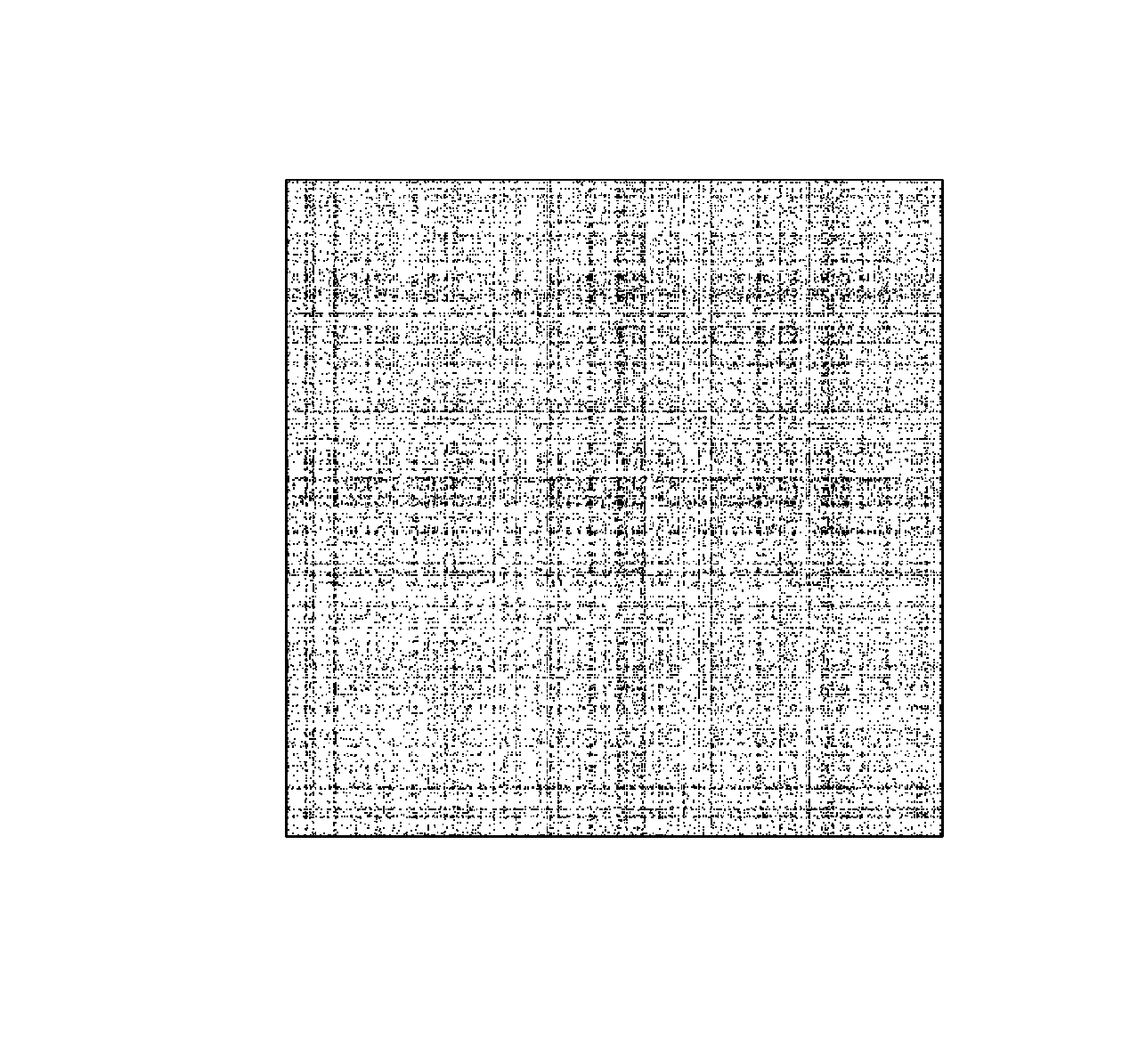}
\includegraphics[trim=20mm 20mm 20mm 20mm, clip, width=0.4\textwidth]{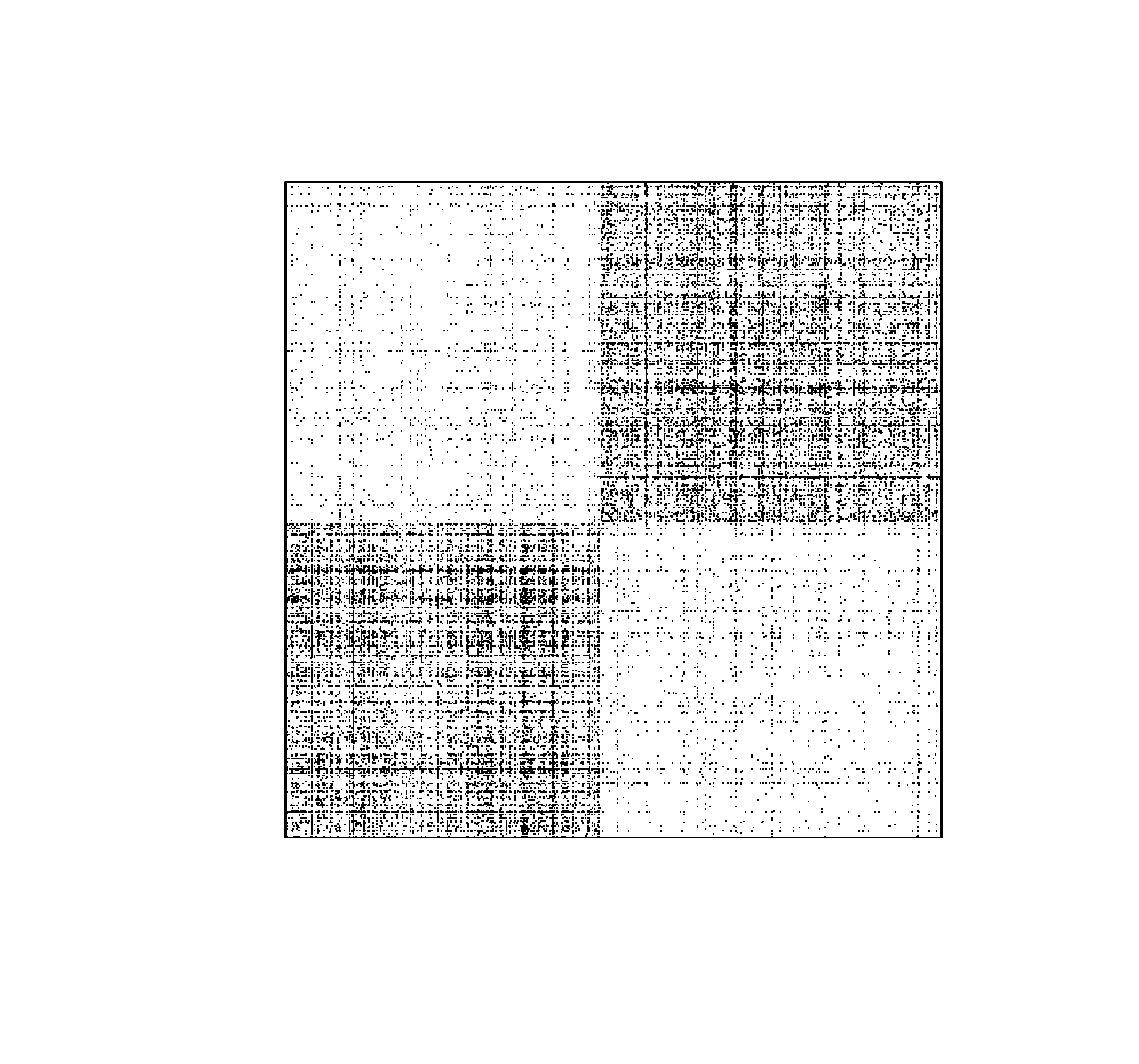}
\caption{Connectivity of political blogs\label{fig:connectivity}. Left panel: plot of the adjacency matrix when the nodes are not grouped. Right panel: plot of the adjacency matrix when the nodes are grouped according to political ideology.}
\end{figure}

Table \ref{tab:tab1} summarizes the results of Algorithm \ref{algo:refine} and its simplified version on this network with four different initialization methods, as well as the performances of directly applying the four methods on the dataset.
The average degree of the network $\bar d$ is 27, which is used as the tuning parameter for regularized NSC.
For regularized USC, we set $\tau$ equals to twice the average degree, leading to the removal of 196 most connected nodes.
The result of directly applying any of the four spectral clustering based initializations was unsatisfactory with at least 30\% nodes misclassified.
Despite the unsatisfactory performance of the initializers, \prettyref{algo:refine} and its simplified version are able to significantly reduce 
the number of misclassified nodes except for the case of NSC$(0)$, and the performance of the two
are close to each other regardless of the initialization method.

\begin{table}
\begin{small}
	\begin{tabular}{c|c|c|c|c|c|c|c|c|c|c|c|c}
	\hline \noalign{\smallskip}
	Initialization                & \multicolumn{3}{c|}{USC$(\infty)$}  & \multicolumn{3}{c|}{USC$(2\bar d)$}    & \multicolumn{3}{c|}{NSC$(0)$}     & \multicolumn{3}{c}{NSC$(\bar d)$}     \\ \hline
	Refinement                  & NA & Algo\ref{algo:refine} & Simple & NA & Algo\ref{algo:refine} & Simple & NA & Algo\ref{algo:refine} & Simple & NA & Algo\ref{algo:refine} & Simple \\ \hline
	\begin{tabular}[c]{@{}l@{}}No. of nodes\\ misclassified\end{tabular} & 383         &  116  & 115 & 583         & 307   & 294   & 579         &  585  & 581    & 308         &   86  & 87     \\ \hline
	\end{tabular}
\end{small}
\caption{Performance on the political blog dataset. 
``NA'' stands for direct application of the initialization method on the whole dataset; 
``Algo \ref{algo:refine}'' stands for the application of  \prettyref{algo:refine} with $\sigma^0$ being the labeled initialization method; 
``Simple'' stands for the application of the simplified version of \prettyref{algo:refine} with $\sigma^0$ being the labeled initialization method.
\label{tab:tab1}}
\end{table}

An interesting observation is that if we apply the refinement scheme multiple times,
the number of misclassified nodes keeps decreasing until convergence and the further reduction of misclassification proportion compared to a single refinement can be sizable. 
Figure \ref{fig:performance} plots the numbers of misclassified nodes for multiple iterations of refinement via the simplified version of \prettyref{algo:refine}. 
We are able to achieve 61, 58 or 63 misclassified nodes out of 1222 depending on which initialization method is used. 
For the three initialization methods included in the figure, the number of misclassified nodes converges within several iterations. 
NSC with $\tau=0$ is not included in Figure \ref{fig:performance} due to the relatively inferior initialization, but its error also converges to around 60/1222 after 20 iterations.
For comparison, state-of-the-art method such as SCORE \cite{jin2015fast} was reported to achieve a comparable error of 58/1222. 
It is worth noting that SCORE was designed under the setting of  degree-corrected stochastic block model, which fits the current dataset better than SBM due to the presence of hubs and low-degree nodes. 
The regularized spectral clustering implemented by \cite{qin2013regularized}, which was also designed under the degree-corrected stochastic block model, was reported to have an error of $(80\pm 2)/1222$. The semi-definite programming method by \cite{cai14} achieved 63/1222.

\begin{figure}
\centering
\includegraphics[trim=0mm 8mm 0mm 2mm, clip, width=\textwidth]{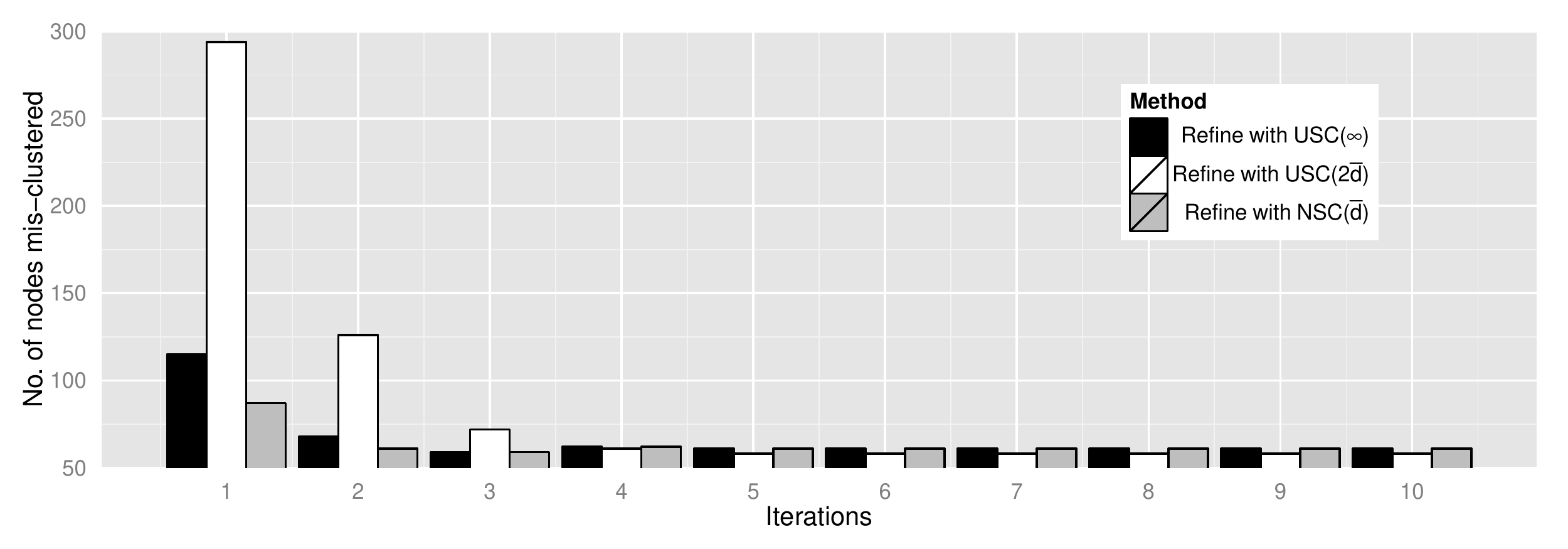}
\caption{Number of misclassified nodes vs.~number of refinement scheme applied\label{fig:performance}}
\end{figure}

To summarize, our algorithm leads to significant performance improvement over several popular spectral clustering based methods on the political blog dataset. With repeated refinements, it demonstrates competitive performance even when compared with methods designed for models that better fit the current dataset.


\section{Discussion}
\label{sec:discussion}
In this section, we discuss a few important issues related to the methodology and theory we have presented in the previous sections.

\subsection{Error bounds when $a\asymp b$ may not hold}
In \prettyref{sec:result}, we established upper bounds on misclassification proportion under the assumption of $a\asymp b$. 
The following theorem shows that slightly weaker upper bounds can be obtained even when $a\asymp b$ does not hold.
To state the result, recall that we assume throughout the paper $\frac{a}{n} \leq 1-\epsilon$ for some numeric constant $\epsilon \in (0,1)$.
\begin{theorem}
\label{thm:upper-a-ll-b}
Suppose as $n\to\infty$, 
$\frac{(a-b)^2}{ak\log k} \to \infty$ and \prettyref{cond:init} is satisfied for $\gamma$ satisfying \eqref{eq:gamma-cond-1}
and $\Theta = \Theta_0(n,k,a,b,\beta)$.
Then for some positive constants $c_\epsilon$ and $C_\eps$ that depend only on $\epsilon$,
for any sufficiently small constant $\epsilon_0 \in (0, c_\epsilon)$, if we replace the definition of $t_u$'s in \eqref{eq:t-set} with 
\begin{align}
\label{eq:t-set-cap}
t_u = \pth{\frac{1}{2}\log\frac{\wh{a}_u (1-\wh{b}_u/n)}{\wh{b}_u (1-\wh{a}_u/n)} } \wedge \log\frac{1}{\epsilon_0/2},
\end{align}
then we have
\begin{equation}
	\label{eq:upper-bound-a-ll-b}
\begin{aligned}
& \sup_{(B,\sigma)\in \Theta}\Prob_{B,\sigma}\sth{
\loss(\sigma, \wh{\sigma}) \geq
\exp\pth{-(1-C_\eps\epsilon_0) \frac{n I^*}{2}}
} \to 0,  & \quad \mbox{if $k = 2$},\\
& \sup_{(B,\sigma)\in \Theta}\Prob_{B,\sigma}\sth{
\loss(\sigma, \wh{\sigma}) \geq
 \exp\pth{-(1-C_\eps\epsilon_0) \frac{n I^*}{\beta k}}
} \to 0, & \quad \mbox{if $k \geq 3$},
\end{aligned}
\end{equation}
where $I^*$ is defined as in \eqref{eq:I}.
In particular, we can set $C_\eps = \frac{10}{3}\frac{2-\epsilon}{\frac{\eps}{2}\log\frac{2}{\eps}}$ and $c_\eps = \min(\frac{1}{10C_\eps}, \frac{\eps}{2-\eps})$.

If in addition \prettyref{cond:init} is satisfied for $\gamma$ satisfying both \eqref{eq:gamma-cond-1} and \eqref{eq:gamma-cond-2} and $\Theta = \Theta(n,k,a,b,\lambda,\beta;\alpha)$, then the same conclusion holds for $\Theta = \Theta(n,k,a,b,\lambda,\beta;\alpha)$.
%
\end{theorem}

Compared with the conclusion \eqref{eq:upper-bound} in \prettyref{thm:upper}, the vanish sequence $\eta$ in the exponent of the upper bound is replaced by $C_\eps \epsilon_0$, which is guaranteed to be smaller than $\min(0.1, \frac{2}{\log(2/\eps)} )$ and can be driven to be arbitrarily small by decreasing $\eps_0$.
To achieve this, the $t_u$'s used in defining the penalty parameters in the penalized neighbor voting step need to be truncated at the value $\log\frac{1}{\epsilon_0/2}$.


\subsection{Implications of the results}

We now discuss some implications of the results in Theorems \ref{thm:upper} -- \ref{thm:upper-a-ll-b}. 

When using USC as initialization for Algorithm \ref{algo:refine}, we obtain the following results by combining Theorem \ref{thm:upper}, Theorem \ref{thm:ini1} and Theorem \ref{thm:upper-a-ll-b}.
Recall that $\bar{d}$ is the average degree of nodes in $A$ defined in (\ref{eq:ave-deg}).

\begin{theorem}\label{cor:upper1}
Consider Algorithm \ref{algo:refine} initialized by $\sigma^0$ with USC$(\tau)$ with $\tau=C \bar{d}$ 
for some sufficiently large constant $C > 0$. 
If as $n\rightarrow\infty$, $a\asymp b$ and
\begin{align}
	\label{eq:cor-upper1-1}
\frac{(a-b)^2}{ak^3\log k}\rightarrow\infty,
\end{align}
then there is a sequence $\eta\rightarrow 0$ such that
\eqref{eq:upper-bound} holds
with $\Theta = \Theta_0(n,k,a,b,\beta)$. 
If as $n\rightarrow\infty$, $a\asymp b$ and
\begin{align}
	\label{eq:cor-upper1-2}
\frac{\lambda^2}{ak(\log k+a/(a-b))}\rightarrow\infty,	
\end{align}
then \eqref{eq:upper-bound} holds for $\Theta=\Theta(n,k,a,b,\lambda,\beta;\alpha)$. 
If for either parameter space, $a\asymp b$ may not hold but $k$ is fixed and  \eqref{eq:cor-upper1-1} or \eqref{eq:cor-upper1-2} holds respectively,
then \eqref{eq:upper-bound-a-ll-b} holds 
as long as $t_u$ is replaced by (\ref{eq:t-set-cap}) in Algorithm \ref{algo:refine}.
\end{theorem}

Compared with Theorem \ref{thm:minimax}, the minimax optimal performance is achieved under mild conditions.
Take $\Theta = \Theta_0(n,k,a,b,\beta)$ for example.
For any fixed $k$,
the minimax optimal misclassification proportion is achieved with high probability only under the additional condition of $a\asymp b$.
In addition, weak consistency is achieved for fixed $k$ as long as $\frac{(a-b)^2}{a}\to \infty$, regardless of the behavior of $\frac{a}{b}$.
This condition is indeed necessary and sufficient for weak consistency. 
See, for instance, \cite{mossel2012stochastic, mossel2013proof,yun2014community,yezhang15}.
To achieve strong consistency for fixed $k$, it suffices to ensure $\ell(\sigma,\wh\sigma) < \frac{1}{n}$ and Theorem \ref{cor:upper1} implies that it is sufficient to have
\begin{align}
	\label{eq:strong-suff}
\liminf_{n\to\infty} \frac{nI^*}{2 \log n} > 1, ~~\mbox{when $k=2$;}\quad
\liminf_{n\to\infty} \frac{nI^*}{\beta k \log n } > 1,
~~\mbox{when $k\geq 3$,}
\end{align}
regardless of the behavior of $\frac{a}{b}$.
On the other hand, \prettyref{thm:minimax} shows that it is impossible to achieve strong consistency if
\begin{align}
	\label{eq:strong-nece}
\limsup_{n\to\infty} \frac{nI^*}{2 \log n} < 1, ~~\mbox{when $k=2$;}\quad
\limsup_{n\to\infty} \frac{nI^*}{\beta k \log n } < 1,
~~\mbox{when $k\geq 3$.}
\end{align}
\rev{When $\frac{a}{n} = o(1)$, $nI^* = (1+o(1))(\sqrt{a}-\sqrt{b})^2$ and so one can replace ${nI^*}$ in \eqref{eq:strong-suff} -- \eqref{eq:strong-nece} with ${(\sqrt{a}-\sqrt{b})^2}$.}
In the literature, \citet{abbe2014exact}, \citet{mossel2014consistency} and \citet{hajek2014achieving} obtained comparable strong consistency results via efficient algorithms for the special case of two communities of equal sizes, \ie, $k=2$ and $\beta = 1$. 
\rev{Under the additional assumption of $a\asymp b\asymp {\log n}$,
\citet{hajek2015achieving} later achieved the result via efficient algorithm for the case of fixed $k$ and $\beta = 1$, and \citet{abbe2015community} investigated the case of fixed $k$ and $\beta \geq 1$.
In comparison, our result holds for any fixed $k$ and any $\beta \geq 1$ without assuming $a\asymp b \asymp \log n$.}
\rev{In the weak consistency regime, in terms of misclassification proportion, for the special case of $k=2$ and $\beta = 1$, \citet{yun2014accurate} achieved the optimal rate for $\Theta_0(n,2,a,b,1)$ when $a \asymp b \asymp a-b$, while the error bounds in other papers are typically off by a constant multiplier on the exponent.
In comparison, Theorem \ref{cor:upper1} provides optimal results \eqref{eq:upper-bound} and near optimal results \eqref{eq:upper-bound-a-ll-b} for a much broader class of models under much weaker conditions.}
Last but not least, our algorithm can provably achieve strong consistency and minimax optimal performance even for growing $k$, which to our limited knowledge, is the first in the literature.


The performance of Algorithm \ref{algo:refine} initialized by NSC can be summarized as the following theorem by combining Theorem \ref{thm:upper}, Theorem \ref{thm:ini2} and Theorem \ref{thm:upper-a-ll-b}.
In this case, the sufficient condition for achieving minimax optimal performance is slightly stronger than when USC is used for initialization.
\begin{theorem}\label{cor:upper2}
Consider Algorithm \ref{algo:refine} initialized by $\sigma^0$ with NSC$(\tau)$ with $\tau=C \bar{d}$ 
for some sufficiently large constant $C >0$. 
If as $n\rightarrow\infty$, $a\asymp b$ and
\begin{align}
	\label{eq:cor-upper2-1}
\frac{(a-b)^2}{ak^3\log k\log a}\rightarrow\infty,
\end{align}
then there is a sequence $\eta\rightarrow 0$ such that
\eqref{eq:upper-bound} holds
with $\Theta = \Theta_0(n,k,a,b,\beta)$. 
If as $n\rightarrow\infty$, $a\asymp b$ and
\begin{align}
	\label{eq:cor-upper2-2}
\frac{\lambda^2}{ak\log a(\log k+a/(a-b))}\rightarrow\infty, 	
\end{align}
then \eqref{eq:upper-bound} holds for $\Theta=\Theta(n,k,a,b,\lambda,\beta;\alpha)$. 
If for either parameter space, $a\asymp b$ may not hold but $k$ is fixed and  \eqref{eq:cor-upper2-1} or \eqref{eq:cor-upper2-2} holds respectively,
then \eqref{eq:upper-bound-a-ll-b} holds 
as long as $t_u$ is replaced by (\ref{eq:t-set-cap}) in Algorithm \ref{algo:refine}.
\end{theorem}

Last but not least, we would like to point out that when the key parameters $a$ and $b$ are known, 
we can obtain the desired performance guarantee under weaker conditions as summarized in the following theorem.

\begin{theorem}[The case of known $a,b$]
\label{thm:upper-bound-ab-known}
Suppose $a, b$ are known.
Consider \prettyref{algo:refine} initialized by $\sigma^0$ with USC($\tau$) with $\tau = C a$ for some sufficiently large constant $C > 0$ and $\wh{a}_u = a$, $\wh{b}_u = b$ in \eqref{eq:ab-esti} for all $u\in [n]$.
If as $n\to\infty$, $a\asymp b$ and
\begin{align}
	\label{eq:ab-known-cond1}
	\frac{(a-b)^2}{ak^3} \to \infty,
\end{align}
then there is a sequence $\eta \to 0$ such that \eqref{eq:upper-bound} holds with $\Theta = \Theta_0(n,k,a,b,\beta)$.
If as $n\to\infty$, $a\asymp b$ and
\begin{align}
		\label{eq:ab-known-cond2}
\frac{\lambda^2}{ak}\to\infty,
\end{align}
then \eqref{eq:upper-bound} holds with $\Theta = \Theta(n,k,a,b,\lambda,\beta;\alpha)$. 
If for either parameter space without assuming $a\asymp b$,  \eqref{eq:ab-known-cond1} or \eqref{eq:ab-known-cond2} holds respectively, then \eqref{eq:upper-bound-a-ll-b} holds 
if in addition $t_u$ is replaced by \eqref{eq:t-set-cap}.

If instead NSC($\tau$) is used for initialization with $\tau = C a$ for some sufficiently large constant $C > 0$, then the above conclusions hold if we replace \eqref{eq:ab-known-cond1} with $\frac{(a-b)^2}{ak^3\log a}\to\infty$ and \eqref{eq:ab-known-cond2} with $\frac{\lambda^2}{ak\log a}\to\infty$, respectively.
\end{theorem}

\subsection{Potential future research problems}


\paragraph{Simplified version of \prettyref{algo:refine} and iterative refinement} 
In simulation studies, we experimented a simplified version of \prettyref{algo:refine} (with precise description as \prettyref{algo:simple} in appendix) and showed that it provided similar performance to \prettyref{algo:refine} on simulated datasets. 
Moreover, for the political blog data, we showed that iterative application of this simplified refinement scheme kept driving down the number of misclassified nodes till convergence.
It is of great interest to see if comparable theoretical results to \prettyref{thm:upper} could be established for the simplified and/or the iterative version, 
and if the iterative version converges to a local optimum of certain objective function for community detection.
Though answering these intriguing questions is beyond the scope of the current paper, we think it can serve as an interesting future research problem.
 

\paragraph{Data-driven choice of $k$}

The knowledge of $k$ is assumed and is used in both methodology and theory of the present paper. 
Date-driven choice of $k$ is of both practical importance and contemporary research interest, and researchers have proposed various ways to achieve this goal for stochastic block model, including cross-validation \citep{chen2014network}, Tracy--Widom test \citep{lei2014goodness}, information criterion \citep{saldana2014many}, likelihood ratio test \citep{wang2015likelihood}, etc. 
Whether these methods are optimal and whether it is possible to select $k$ in a statistically optimal way remains an important open problem.

\paragraph{More general models}

The results in this paper cover a large range of parameter spaces for stochastic block models and we show the competitive performance of the proposed algorithm both in theory and on numerical examples. 
Despite its popularity, stochastic block model has its own limits for modeling network data.
Therefore, 
an important future research direction is to design computationally feasible algorithms that can achieve statistically optimal performance for more general network models, such as degree-corrected stochastic block models.


\section{Proofs of main results}
\label{sec:proof}

The main result of the paper, \prettyref{thm:upper}, is proved in Section \ref{sec:hahaha}. Theorem \ref{thm:ini1} and Theorem \ref{thm:ini2} are proved in Section \ref{sec:hehehe} and Section \ref{sec:xixixi} respectively. The proofs of the remaining results, together with some auxiliary lemmas, are given in the appendix.

\subsection{Proof of \prettyref{thm:upper}}\label{sec:hahaha}



We first state a lemma that guarantees the accuracy of parameter estimation in Algorithm \ref{algo:refine}.
\begin{lemma}
\label{lmm:ab-est}
Let $\Theta = \Theta(n,k,a,b,\lambda,\beta;\alpha)$.
Suppose as $n\to\infty$, $\frac{(a-b)^2}{ak} \to \infty$ and \prettyref{cond:init} holds with $\gamma$ satisfying \eqref{eq:gamma-cond-1} and \eqref{eq:gamma-cond-2}.
Then there is a sequence $\eta\to 0$ as $n\to\infty$ and a constant $C > 0$ such that
\begin{align}
	\label{eq:B-est-upper}
 \min_{u\in [n]}\inf_{(B,\sigma)\in \Theta} \Prob\sth{
\min_{\pi\in S_k}\max_{i, j\in [k]} |\wh{B}_{ij}^u - B_{\pi(i) \pi(j)}| \leq \eta \pth{\frac{a-b}{n}} } \geq 1 - C n^{-(1+\delta)}.
\end{align}
For $\Theta = \Theta_0(n,k,a,b,\beta)$, 
the conclusion \eqref{eq:B-est-upper} continues to hold even when the assumption \eqref{eq:gamma-cond-2} is dropped.
\end{lemma}

\begin{proof}
$1^\circ$ Let $\Theta = \Theta(n,k,a,b,\lambda,\beta;\alpha)$.
For any community assignments $\sigma_1$ and $\sigma_2$, define
\begin{align}
\label{eq:loss-0}
\loss_0(\sigma_1,\sigma_2) = \frac{1}{n}\sum_{u=1}^n\indc{\sigma_1(u)\neq\sigma_2(u)}.
\end{align}
Fix any $(B,\sigma)\in \Theta$ and $u\in [n]$.
Define event 
\begin{align}
E_u = \sth{\loss_0(\pi_u(\sigma), \sigma^0_u) \leq \gamma}.
\end{align}
To simplify notation, assume that $\pi_u = \mathrm{Id}$ is the identity permutation. 

Fix any $i\in [k]$. On $E_u$, 
\begin{align}
	\label{eq:Ci-cond}
n_i\geq |\wt\calC_i^u\cap \calC_i| \geq n_i-\gamma_1 n,\,\,
|\wt\calC_i^u\cap \calC_i^c|\leq \gamma_2 n,
\quad \mbox{where $\gamma_1, \gamma_2\geq 0$ and $\gamma_1 + \gamma_2 \leq \gamma$}.
\end{align}
Let $\calC_i'$ be any deterministic subset of $[n]$ such that \eqref{eq:Ci-cond} holds with $\wt\calC_i^u$ replaced by $\calC_i'$. 
By definition, there are at most 
\begin{align*}
\sum_{l=0}^{\gamma n} {n_i\choose l} \sum_{m=0}^{\gamma n}{n-n_i\choose m}
& \leq (\gamma n + 1)^2 \pth{e n_i \over \gamma n}^{\gamma n} 
\pth{e n\over \gamma n}^{\gamma n} 
\leq \exp\sth{2\log(\gamma n + 1) + 2\gamma n\log \frac{e}{\gamma}}\\
& \leq \exp\sth{C_1 \gamma n\log \frac{1}{\gamma}}
\end{align*}
different subsets with this property where $C_1 > 0$ is an absolute constant.
Let $\calE_i'$ be the edges within $\calC_i'$. 
Then $|\calE_i'|$ consists of independent Bernoulli random variables, where at least $(1-\beta \gamma k)^2$ proportion of them follow the Bern$(B_{ii})$ distribution, at most $(\beta\gamma k)^2$ proportion that are stochastically smaller than Bern$(\frac{\alpha a}{n})$ and stochastically larger than Bern$(\frac{a}{n})$,
and at most $2\beta\gamma k$ proportion are stochastically smaller than Bern$(\frac{b}{n})$. 
Therefore, we obtain that
\begin{align}
	\label{eq:mean-range}
(1-\beta\gamma k)^2 B_{ii} + (\beta\gamma k)^2 \frac{a}{n}
\leq 
\Expect\qth{\frac{|\calE_i'|}{\frac{1}{2}|\calC_i'|(|\calC_i'| - 1)}}
\leq 
\max_{t\in [0,\beta\gamma k]}\sth{(1-t)^2 B_{ii} + t^2 \frac{\alpha a}{n} + 2t\frac{b}{n}}.
\end{align}
Note that the LHS is $(1 - (2+o(1))\beta \gamma k)B_{ii}$.
On the other hand, under condition \eqref{eq:gamma-cond-2},
the RHS is attained at $t = 0$ and equals $B_{ii}$ exactly.
Thus, we conclude that 
\begin{align}
	\label{eq:bias-control}
\left| \Expect\qth{\frac{|\calE_i'|}{\frac{1}{2}|\calC_i'|(|\calC_i'| - 1)}} - B_{ii} \right| \leq C\beta\gamma k \frac{\alpha a}{n}  = \eta'\pth{\frac{a-b}{n}}
\end{align}
for some $\eta'\to 0$ that depends only on $a,k,\alpha,\beta$ and $\gamma$,
where the last inequality is due to \eqref{eq:gamma-cond-2}.

On the other hand, by Bernstein's inequality, for any $t > 0$, 
\begin{align*}
\Prob\sth{ \left| |\calE_i'| - \Expect|\calE_i'| \right| > t} \leq 
2\exp\sth{-\frac{t^2}{2(\frac{1}{2}(n_i+\gamma n)^2 \frac{\alpha a}{n} + \frac{2}{3} t)}}.
\end{align*}
Let
\begin{align*}
t^2 & = (n_i+\gamma n)^2 \frac{\alpha a}{n} (C_1\gamma n\log\gamma^{-1} + (3+\delta)\log n) \vee (2C_1\gamma n \log\gamma^{-1} + 2(3+\delta)\log n)^2 \\
& \lesssim \pth{\frac{n}{k}\sqrt{a\gamma \log \gamma^{-1}} + \gamma n\log \gamma^{-1} }^2,
\end{align*}
where we the second inequality holds since $\frac{\log x}{x}$ is monotone decreasing as $x$ increases and so $\gamma\log \gamma^{-1} \geq \frac{1}{n}\log n$ for any $\gamma \geq \frac{1}{n}$, which is the case of most interest since $\gamma < \frac{1}{n}$ leads to $\gamma = 0$ and so the initialization is already perfect.
Even when $\gamma = 0$, we can still continue to the following arguments by replacing every $\gamma$ with $\frac{1}{n}$ and all the steps continue to hold.
Thus, we obtain that for positive constant $C_{\alpha,\beta,\delta}$ that depends only on $\alpha,\beta$ and $\delta$,
\begin{align}
\Prob\sth{\left| |\calE_i'| - \Expect|\calE_i'| \right| > 
C_{\alpha,\beta,\delta} \pth{\frac{n}{k}\sqrt{a\gamma \log \gamma^{-1}} + \gamma n\log \gamma^{-1}} } 
\leq \exp\sth{-C_1\gamma n\log \gamma^{-1}}n^{-(3+\delta)}.
\end{align}
Thus, with probability at least $1-\exp\sth{-C_1\gamma n\log \gamma^{-1}}n^{-(3+\delta)}$,
\begin{align}
	\label{eq:error-control}
\left| \frac{|\calE_i'|}{\frac{1}{2}|\calC_i'|(|\calC_i'|-1)} - \Expect\frac{|\calE_i'|}{\frac{1}{2}|\calC_i'|(|\calC_i'|-1)} \right|	
\leq C_{\alpha, \beta,\delta}\pth{\frac{k}{n}\sqrt{a\gamma \log \gamma^{-1}} + \frac{k^2\gamma \log \gamma^{-1}}{n}} = \eta'\pth{\frac{a-b}{n}},
\end{align}
where $\eta'\to 0$ depends only on $a,k,\alpha, \beta,\gamma$ and $\delta$.
Here, the last inequality holds since 
\begin{align*}
k\sqrt{a\gamma \log \gamma^{-1}}= \sqrt{ak} \sqrt{k\gamma \log \gamma^{-1}},
\end{align*}
where $\sqrt{ak}\ll a-b$ since $\frac{(a-b)^2}{ak}\to \infty$ and $k\gamma \log \gamma^{-1} = O(1)$, and
\begin{align*}
k^2 \gamma \log \gamma^{-1} = k\gamma\log\gamma^{-1} \cdot k \lesssim k \ll \frac{(a-b)^2}{a} \lesssim a-b.
\end{align*}
We combine \eqref{eq:bias-control} and \eqref{eq:error-control} and apply the union bound 
to obtain that for a sequence $\eta\to 0$ that depends only on $a,k,\alpha, \beta, \gamma$ and $\delta$, with probability at least $1-n^{-(3+\delta)}$
\begin{align}
\left| \frac{|\wt\calE_i^u|}{\frac{1}{2}|\wt\calC_i^u|(|\wt\calC_i^u|-1)} - B_{ii} \right|	\leq \eta \pth{a-b\over n}.
\end{align}
The proof for $B_{ij}$ estimation is analogous and hence is omitted. 
A final union bound on $i,j\in [k]$ leads to the desired claim since all the constants and vanishing sequences in the above analysis depend only on $a,b,k,\alpha,\beta,\gamma$ and $\delta$, but not on $u$, $B$ or $\sigma$.

\medskip

$2^\circ$
If $\Theta = \Theta_0(n,k,a,b,\beta)$, then condition \eqref{eq:gamma-cond-2} on $\gamma$ is no longer needed. 
This is because \eqref{eq:mean-range} can be replaced by
\begin{align*}
& \min_{t\in [0,\beta\gamma k]}\sth{(1-t)^2 \frac{a}{n} + 2t(1-t)\frac{b}{n} + t^2\frac{b}{n}}\\
& \hskip 4em \leq 
	\Expect\qth{\frac{|\calE_i'|}{\frac{1}{2}|\calC_i'|(|\calC_i'| - 1)}}
	\leq 
	\max_{t\in [0,\beta\gamma k]}\sth{(1-t)^2 \frac{a}{n} + t^2 \frac{a}{n} +2t(1-t)\frac{b}{n}},
\end{align*}
where the LHS equals $\frac{a}{n} - (1-\beta\gamma k(1+o(1)))\frac{a-b}{n} = \frac{a}{n} + o(\frac{a-b}{n})$ and the RHS equals $\frac{a}{n}$.
Thus, no additional condition is needed to guarantee \eqref{eq:bias-control} in the foregoing arguments. This completes the proof.
\end{proof}

The next two lemmas establish the desired error bound for the node-wise refinement.

\begin{lemma}
	\label{lmm:refine-nb}
Let $\Theta_0$ be defined as in \eqref{eq:para-space0} and $k\geq 2$.
Suppose as $n\to\infty$, $\frac{(a-b)^2}{ak}\to\infty$ and $a\asymp b$.
If there exists two sequences $\gamma = o(1/k)$ and $\eta = o(1)$, constants $C,\delta > 0$ and permutations $\sth{\pi_u}_{u=1}^n\subset S_k$ such that 
\begin{equation}
	\label{eq:refine-cond}
\begin{aligned}
\inf_{(B,\sigma)\in \Theta_0} 
\min_{u\in [n]}
	\Prob\sth{
	\loss_0(\pi_u(\sigma), \sigma_u^0) \leq \gamma,~
	|\wh{a}_u -a| \leq \eta(a-b),~ 
	|\wh{b}_u - b| \leq \eta (a-b)} 
\geq 1 - Cn^{-(1+\delta)}. &	
\end{aligned}
\end{equation}
Then for 
$\wh{\sigma}_u(u)$ defined as in \eqref{eq:refine-nb} with $\rho = \rho_u$ in \eqref{eq:lambda-set}, there is a sequence $\eta' = o(1)$ such that for $k=2$,
$$
\sup_{(B,\sigma)\in \Theta_0}\max_{u\in [n]}\Prob\sth{\wh\sigma_u(u) \neq \pi_u(\sigma(u))} \leq 
(k-1)\exp\sth{-(1-\eta')\frac{nI^*}{2}}
+ C n^{-(1+\delta)},
$$
and for $k\geq 3$,
$$
\sup_{(B,\sigma)\in \Theta_0}\max_{u\in [n]}\Prob\sth{\wh\sigma_u(u) \neq \pi_u(\sigma(u))} \leq 
(k-1)\exp\sth{-(1-\eta')\frac{nI^*}{\beta k}}
+ C n^{-(1+\delta)}.
$$
\end{lemma}
\begin{proof}
In what follows, let $E_u$ denote the event in \eqref{eq:refine-cond}. 
For the sake of brevity, we let $p=a/n$, $q=b/n$, $\wh{p}_u = \wh{a}_u/n$ and $\wh{q}_u = \wh{b}_u/n$. Moreover, let $\sigma_u = \pi_u(\sigma)$,
$n_i = |\{v: \sigma_u(v) =i\}|$, 
$m_i = |\{v:\sigma^0_u(v)=i\}|$ and
$m_i'= |\{v:\sigma^0_u(v)= \sigma_u(v)= i\}|$. 
Without loss of generality, let $\sigma_u(u)=1$. 

Then we have
\begin{align}
	\label{eq:pl-decomp}
\Prob\sth{\wh\sigma_u(u) \neq 1 ~\mbox{and}~ E_u} 
\leq \sum_{l\neq 1}
\Prob\sth{E_u \mbox{~and~} \sum_{\sigma_u(v) = l}A_{uv}  - \sum_{\sigma_u(v) = 1}A_{uv} \geq  \rho_u(m_l - m_1) } = \sum_{l\neq 1} p_l.
\end{align}
Now we bound each $p_l$. By the independence structure and Chernoff bound, we have
\begin{eqnarray}
\nonumber p_l &\leq& \mathbb{E}\Big\{\exp\left(-t_u \rho_u(m_l-m_1)\right)(qe^{t_u}+1-q)^{m_l'}(pe^{t_u}+1-p)^{m_l-m_l'}\nonumber\\
& & ~~~~~ (pe^{-t_u}+1-p)^{m_1'}(qe^{-t_u}+1-q)^{m_1-m_1'}\indc{E_u}\Big\} \\
\label{eq:fix1} &\leq& \mathbb{E}\left\{\exp\left(-t_u\rho_u(m_l-m_1)\right)(qe^{t_u}+1-q)^{m_l}(pe^{-{t_u}}+1-p)^{m_1}\indc{E_u}\right\} \\
\label{eq:fix2} && \times\mathbb{E}\left\{\left(\frac{pe^{t_u}+1-p}{qe^{t_u}+1-q}\right)^{m_l-m_l'}\left(\frac{qe^{-t_u}+1-q}{pe^{-t_u}+1-p}\right)^{m_1-m_1'}\indc{E_u}\right\}.
\end{eqnarray}
We are going to give bounds for the terms in (\ref{eq:fix1}) and (\ref{eq:fix2}) respectively. Before doing that, we need some preparatory inequalities. Define $t^*$ through the equation
$$e^{t^*}=\sqrt{\frac{p(1-q)}{q(1-p)}}.$$
Then, on the event $E_u$, 
\begin{equation}
e^{t_u-t^*}+e^{t^*-t_u} \leq \exp(C_1\eta),\label{eq:repair1}
\end{equation}
for some constant $C_1>0$. Moreover, 
\begin{equation}
|e^{t_u}-1|\vee|e^{-t_u}-1| \leq C_2\frac{p-q}{p} = C_2\frac{a-b}{a},\label{eq:repair2}
\end{equation}
for some constant $C_2>0$. 
Therefore, for the term in \eqref{eq:fix1}, on the event $E_u$, 
\begin{eqnarray}
\nonumber && \exp\left(-t_u\rho_u(m_l-m_1)\right)(qe^{t_u}+1-q)^{m_l}(pe^{-t_u}+1-p)^{m_1} \\
\label{eq:fix3} &=& \exp\left(-{t_u}{\rho_u}(m_l-m_1)\right)\left(qe^{t_u}+1-q\right)^{(m_l-m_1)/2}\left(pe^{-t_u}+1-p\right)^{(m_1-m_l)/2} \\
\label{eq:fix4} && \times \left(qe^{t_u}+1-q\right)^{(m_1+m_l)/2}\left(pe^{-t_u}+1-p\right)^{(m_1+m_l)/2}.
\end{eqnarray}
By \eqref{eq:repair1}, the term in (\ref{eq:fix4}) is upper bounded by
\begin{eqnarray*}
&& \left(pq+(1-p)(1-q)+\sqrt{pq}\sqrt{(1-p)(1-q)}(e^{t_u-t^*}+e^{t^*-t_u})\right)^{\frac{m_1+m_l}{2}} \\
&& ~~ \leq \exp\left(-(1+o(1))\frac{m_1+m_l}{2}I^*\right) 
\leq \exp\left(-(1+o(1))\frac{n_1+n_l}{2}I^*\right).
\end{eqnarray*}
By \eqref{eq:repair2}, the term in (\ref{eq:fix3}) is upper bounded by
\begin{eqnarray*}
&&  \exp\left(-t_u{\rho_u}(m_l-m_1)\right)\left(qe^{t_u}+1-q\right)^{(m_l-m_1)/2}\left(pe^{-t_u}+1-p\right)^{(m_1-m_l)/2} \\
&=& \exp\left(\frac{m_1-m_l}{2}\left(\log\frac{pe^{-t_u}+1-p}{qe^{t_u}+1-q}-\log\frac{\wh{p}_ue^{-t_u}+1-\wh{p}_u}{\wh{q}_ue^{t_u}+1-\wh{q}_u}\right)\right) \\
&\leq& \exp\left(\frac{|m_1-m_l|}{2}\left(|e^{-t_u}-1||\wh{p}_u-p|+|e^{t_u}-1||\wh{q}_u-q|\right)\right) \\
&\leq& \exp\left(o\left(\frac{n}{k}\frac{(p-q)^2}{p}\right)\right) \\
&=& \exp\left(o(1)\frac{n_1+n_l}{2}I^*\right).
\end{eqnarray*}
Therefore, we can upper bound (\ref{eq:fix1}) as
\begin{equation}
\mathbb{E}\left\{e^{-t_u\rho_u(m_l-m_1)}(qe^{t_u}+1-q)^{m_l}(pe^{-{t_u}}+1-p)^{m_1}\indc{E_u}\right\} \leq \exp\left(-(1+o(1))\frac{n_1+n_l}{2}I^*\right).\label{eq:fix100}
\end{equation}
Now we provide an upper bound for (\ref{eq:fix2}). 
By \eqref{eq:repair2}, on $E_u$,
\begin{eqnarray*}
\frac{pe^{t_u}+1-p}{qe^{t_u}+1-q} = 1+\frac{(p-q)(e^{t_u}-1)}{qe^{t_u}+1-q} 
\leq 1+O\left(\frac{(p-q)^2}{p}\right) 
\leq \exp\left(O\left(\frac{(p-q)^2}{p}\right)\right),
\end{eqnarray*}
and
\begin{eqnarray*}
\frac{qe^{-t_u}+1-q}{pe^{-t_u}+1-p} = 1+\frac{(p-q)(1-e^{-t_u})}{pe^{-t_u}+1-p} 
\leq 1+O\left(\frac{(p-q)^2}{p}\right) 
\leq \exp\left(O\left(\frac{(p-q)^2}{p}\right)\right).
\end{eqnarray*}
Therefore,
\begin{equation}
\mathbb{E}\left\{\left(\frac{pe^{t_u}+1-p}{qe^{t_u}+1-q}\right)^{m_l-m_l'}\left(\frac{qe^{-t_u}+1-q}{pe^{-t_u}+1-p}\right)^{m_1-m_1'}\indc{E_u}\right\}\leq \exp\left(o(1)\frac{n_1+n_l}{2}I^*\right).\label{eq:fix200}
\end{equation}
By combining (\ref{eq:fix100}) and (\ref{eq:fix200}), we have
\begin{equation}
	\label{eq:pl-bound}
p_l\leq \exp\left(-(1+o(1))\frac{n_1+n_l}{2}I^*\right).
\end{equation}
Using (\ref{eq:pl-decomp}), this implies
$$\Prob\sth{\wh\sigma_u(u) \neq 1 ~\mbox{and}~ E_u} \leq (k-1)\exp\left(-(1+o(1))\min_{l\neq 1}\left(\frac{n_1+n_l}{2}\right)I^*\right),$$
and so
$$\Prob\sth{\wh\sigma_u(u) \neq 1} \leq (k-1)\exp\left(-(1+o(1))\min_{l\neq 1}\left(\frac{n_1+n_l}{2}\right)I^*\right)+Cn^{-(1+\delta)}.$$
When $k=2$, $\min_{l\neq 1}\left(\frac{n_1+n_l}{2}\right)=\frac{n}{2}$, and when $k\geq 3$, $\min_{l\neq 1}\left(\frac{n_1+n_l}{2}\right)\geq \frac{n}{\beta k}$.
Thus, the proof is complete.
\end{proof}

\begin{lemma}
	\label{lmm:refine-nb-2}
Let $\Theta$ be defined as in \eqref{eq:para-space} and $k\geq 2$.
Suppose as $n\to\infty$, $\frac{(a-b)^2}{ak}\to\infty$ and $a\asymp b$.
If there exists two sequences $\gamma = o\left(\frac{a-b}{ak}\right)$ and $\eta = o(1)$, constants $C,\delta > 0$ and permutations $\sth{\pi_u}_{u=1}^n\subset S_k$ such that \eqref{eq:refine-cond} holds.
Then for 
$\wh{\sigma}_u(u)$ defined as in \eqref{eq:refine-nb} with $\rho = \rho_u$ in \eqref{eq:lambda-set}, 
the conclusions of \prettyref{lmm:refine-nb} continue to hold.
\end{lemma}
\begin{proof}
The proof is similar to that of \prettyref{lmm:refine-nb} and we use the same notation as there. 
First, we give a bound for $p_l$ defined in (\ref{eq:pl-decomp}). Let $X_j\sim \mbox{Bern}(q)$, $Y_j\sim \mbox{Bern}(p)$ and $Z_j\sim \mbox{Bern}(\alpha p)$, $j\geq 1$, be mutually independent. Then, a stochastic order argument gives
\begin{eqnarray}
\nonumber p_l &\leq& \mathbb{E}\left[\mathbb{P}\left\{\sum_{j=1}^{m_l'}X_j+\sum_{j=1}^{m_l-m_l'}Z_j-\sum_{j=1}^{m_1'}Y_j\geq\rho(m_l-m_1)\text{ and }E_u\Big|A_{-u}\right\}\right] \\
\label{eq:fix11} &\leq& \mathbb{E}\left\{\exp\left(-t_u\rho_u(m_l-m_1)\right)(qe^{t_u}+1-q)^{m_l}(pe^{-t_u}+1-p)^{m_1}\indc{E_u}\right\} \\
\label{eq:fix12} && \times\mathbb{E}\left\{\left(\frac{1}{qe^{t_u}+1-q}\right)^{m_l-m_l'}\left(\frac{1}{pe^{-{t_u}}+1-p}\right)^{m_1-m_1'}(\alpha pe^{t_u}+1-\alpha p)^{m_l-m_l'}\indc{E_u}\right\}.
\end{eqnarray}
Note that the term in (\ref{eq:fix11}) is the same as that in (\ref{eq:fix1}), and thus it can be upper bounded by (\ref{eq:fix100}) as before. 
To bound for (\ref{eq:fix12}), observe that by \eqref{eq:repair2},
$$\frac{1}{qe^{t_u}+1-q}\leq \exp\left(q|e^{t_u}-1|\right)\leq \exp\left(O(p-q)\right),$$
$$\frac{1}{pe^{-t_u}+1-p}\leq \exp\left(Cp|e^{-t_u}-1|\right)\leq  \exp\left(O(p-q)\right)$$
and
$$\alpha pe^{t_u}+1-\alpha p\leq \exp\left(\alpha p|e^{{t_u}}-1|\right)\leq  \exp\left(O(p-q)\right).$$
Thus, under the assumption $\gamma=o\left(\frac{p-q}{kp}\right)$, the term (\ref{eq:fix12}) is bounded by $\exp\left(o(1)\frac{n_1+n_l}{2}I^*\right)$. The remaining proof is the same as that of \prettyref{lmm:refine-nb}.
\end{proof}

Finally, we need a lemma to justify the consensus step in Algorithm \ref{algo:refine}.

\begin{lemma}
	\label{lmm:consensus}
For any community assignments $\sigma$ and $\sigma'$: $[n]\to [k]$, such that for some constant $C \geq 1$
\[
\min_{l\in [k]}|\sth{u:\sigma(u)=l}|,\, \min_{l\in [k]}|\sth{u:\sigma'(u)=l}| \geq \frac{ n}{ C k}, 
\quad
\mbox{and}
\quad 
\min_{\pi\in S_k}\loss_0(\sigma, \pi(\sigma')) < \frac{1}{C k}.
\]
Define map $\xi:[k]\to [k]$ as 
\begin{align}
	\label{eq:label-change}
\xi(i) = \argmax_{l}\left|\{u:\sigma(u) = l\} \cap \{u: \sigma'(u) = i\} \right|,\quad \forall i\in [k].
\end{align}
Then $\xi\in S_k$ and $\loss_0(\sigma, \xi(\sigma')) = \min_{\pi\in S_k}\loss_0(\sigma, \pi(\sigma'))$.
\end{lemma}
\begin{proof}
By the definition in \eqref{eq:label-change}, we obtain
\begin{align*}
\xi = \argmin_{\xi':[k]\to [k]} \loss_0(\sigma, \xi'(\sigma')),\quad
\mbox{and}\quad
\loss_0(\sigma, \xi(\sigma')) \leq \min_{\pi\in S_k}\loss_0(\sigma, \pi(\sigma')) < \frac{1}{C k}.
\end{align*}
Thus, what remains to be shown is that $\xi\in S_k$, \ie, $\xi(l_1)\neq \xi(l_2)$ for any $l_1\neq l_2$. 
To this end, note that if for some $l_1 \neq l_2$, $\xi(l_1) = \xi(l_2)$, then 
there would exist some $l_0\in [k]$ such that for any $l\in [k]$, $\xi(l)\neq l_0$, and so
\begin{align*}
\loss_0(\sigma, \xi(\sigma')) \geq 
\frac{1}{n}\sum_{u:\sigma(u) = l_0}\indc{\sigma(u)\neq \xi(\sigma'(u))} = \frac{|\sth{u:\sigma(u)=l_0}|}{n} \geq \frac{1}{C k}.
\end{align*}
This is in contradiction to the second last display, and hence $\xi\in S_k$. 
This completes the proof.
\end{proof}



\begin{proof}[Proof of \prettyref{thm:upper}]
Let $\Theta = \Theta(n,k,a,b,\lambda,\beta;\alpha)$, and fix any $(B,\sigma)\in \Theta$.
For any $u\in [n]$,
by \prettyref{cond:init} and the fact that $\sigma^0_u$ and $\wh\sigma_u$ differ only at the community assignment of $u$, for $\gamma' = \gamma + 1/n$,
there exists some $\pi_u\in S_k$ such that 
\begin{align}
	\label{eq:thm1-pf-cond1}
\Prob\sth{\loss_0(\sigma, \pi_u^{-1}(\wh\sigma_u)) \leq \gamma'_n} 
\geq 1-C_0 n^{-(1+\delta)}.
\end{align}
Without loss of generality, we assume $\pi_1 = \mathrm{Id}$ is the identity map.
Now for any fixed $u\in \sth{2,\dots, n}$, define map $\xi_u:[k]\to [k]$ as in \eqref{eq:label-change} with $\sigma$ and $\sigma'$ replaced by $\wh\sigma_1$ and $\wh\sigma_u$. 
Then by definition 
\begin{align}
\wh\sigma(u) = \xi_u(\wh\sigma_u(u)).	
\end{align}
In addition, \eqref{eq:thm1-pf-cond1} implies 
with probability at least $1 - Cn^{-(1+\delta)}$, we have 
\begin{align*}
\ell_0(\sigma, \wh\sigma_1)\leq \gamma'
\quad \mbox{and}\quad 
\ell_0(\sigma, \pi_u^{-1}(\wh{\sigma}_u)) \leq \gamma'.
\end{align*}
So the triangle inequality implies 
$\ell_0(\wh\sigma_1, \pi_u^{-1}(\wh{\sigma}_u)) \leq 2\gamma'$ and hence the condition of \prettyref{lmm:consensus} is satisfied.
Thus, \prettyref{lmm:consensus} implies
\begin{align}
	\label{eq:consensus-bd}
\Prob\sth{\xi_u = \pi_u^{-1}} \geq 1 - Cn^{-(1+\delta)}.
\end{align}

When $k\geq 3$,
\prettyref{lmm:ab-est}, \eqref{eq:gamma-cond-1} and \eqref{eq:gamma-cond-2} imply that the condition of \prettyref{lmm:refine-nb-2} is satisfied, which in turn implies that for a sequence $\eta' = o(1)$, 
\begin{align*}
\Prob\sth{\wh\sigma(u) \neq \sigma(u)} 
& = \Prob\sth{\xi_u(\wh\sigma_u(u))\neq \sigma(u)} \\
& \leq \Prob\sth{\xi_u(\wh\sigma_u(u))\neq \sigma(u),\, \xi_u = \pi_u^{-1}} 
+ \Prob\sth{\xi_u \neq \pi_u^{-1}} \\
& \leq \Prob\sth{\wh\sigma_u(u)\neq \pi_u(\sigma(u))} + \Prob\sth{\xi_u \neq \pi_u^{-1}} \\
& \leq Cn^{-(1+\delta)} + (k-1) \exp\sth{-(1-\eta')\frac{nI^*}{\beta k}}  .
\end{align*}
Set
\begin{align}
\eta = \eta' + \beta\sqrt{\frac{ k}{n I^*}} = o(1)
\end{align}
where the last inequality holds since $\frac{\,nI^* }{k} \asymp \frac{(a-b)^2}{ak} \to\infty$.
Thus, Markov's inequality leads to
\begin{align*}
& 
\hskip -4em \Prob\sth{\loss_0(\sigma, \wh\sigma) > 
(k-1) \exp\sth{-(1-\eta)\frac{nI^*}{\beta k}}
}  
\\
& \leq \frac{1}{ (k-1) \exp\sth{-(1-\eta)\frac{nI^*}{\beta k}} } \frac{1}{n}\sum_{u=1}^n \Prob\sth{\wh\sigma(u) \neq \sigma(u)} \\
& \leq \exp\sth{-(\eta-\eta') \frac{nI^*}{\beta k}} 
+ \frac{Cn^{-(1+\delta)}}{(k-1) \exp\sth{-(1-\eta)\frac{nI^*}{\beta k}} } \\
& \leq \exp\sth{-\sqrt{n I^*\over k}} 
+ \frac{Cn^{-(1+\delta)}}{(k-1) \exp\sth{-(1-\eta)\frac{nI^*}{\beta k}} }.
\end{align*}
If $(k-1) \exp\sth{-(1-\eta)\frac{nI^*}{\beta k}} \geq n^{-(1+\delta/2)}$, then
\begin{align*}
\Prob\sth{\loss_0(\sigma, \wh\sigma) > 
(k-1) \exp\sth{-(1-\eta)\frac{nI^*}{\beta k}}
} 
\leq \exp\sth{-\sqrt{n I^* \over k}} + Cn^{-\delta/2} = o(1).
\end{align*}
If $(k-1) \exp\sth{-(1-\eta)\frac{nI^*}{\beta k}} < n^{-(1+\delta/2)}$, then
\begin{align*}
\Prob\sth{\loss_0(\sigma, \wh\sigma) > 
(k-1) \exp\sth{-(1-\eta)\frac{nI^*}{\beta k}}
} = 
\Prob\sth{\loss_0(\sigma, \wh\sigma) > 0}
\leq \sum_{u=1}^n \Prob\sth{\wh\sigma(u)\neq \sigma(u)} ~~~ & \\
\leq n(k-1) \exp\sth{-(1-\eta)\frac{nI^*}{\beta k}} + Cn^{-\delta} 
\leq Cn^{-\delta/2} = o(1). & 
\end{align*}
Here, the second last inequality holds since $\eta > \eta'$ and so 
$ (k-1)\exp\sth{-(1-\eta')nI^*/(\beta k)} < (k-1) \exp\sth{-(1-\eta)nI^*/(\beta k)} < n^{-(1+\delta/2)}$.
We complete the proof for the case of $\Theta(n,k,a,b,\lambda,\beta;\alpha)$ and $k\geq  3$ by noting that $(k-1) \exp\sth{-(1-\eta)\frac{nI^*}{\beta k}}=\exp\left\{-(1-\eta'')\frac{nI^*}{\beta k}\right\}$ for another sequence $\eta''=o(1)$ under the assumption $\frac{(a-b)^2}{ak\log k}\rightarrow\infty$ and no constant or sequence in the foregoing arguments involves $B,\sigma$ or $u$.
When $\Theta = \Theta(n,k,a,b,\lambda,\beta;\alpha)$ and $k = 2$, the foregoing arguments continue to hold with $\beta$ and $k$ replaced with $1$ and $2$ respectively.

When $\Theta = \Theta_0(n,k,a,b,\beta)$, we can run the foregoing arguments with \prettyref{lmm:refine-nb-2} replaced by \prettyref{lmm:refine-nb} to reach the conclusion in \eqref{eq:upper-bound}, which does not require condition \eqref{eq:gamma-cond-2}.
This completes the proof.
\end{proof}

\subsection{Proof of Theorem \ref{thm:ini1}}\label{sec:hehehe}

The following lemma is critical to establish the result of  Theorem \ref{thm:ini1}. Its proof is given in the appendix. Let us introduce the notation $O(k_1,k_2)=\{V\in\mathbb{R}^{k_1\times k_2}: V^TV=I_{k_2}\}$ for $k_1\geq k_2$.

\begin{lemma} \label{lem:sp1}
Consider a symmetric adjacency matrix $A\in\{0,1\}^{n\times n}$ and a symmetric matrix $P\in[0,1]^{n\times n}$ satisfying $A_{uu}=0$ for all $u\in[n]$ and $A_{uv}\sim \text{Bernoulli}(P_{uv})$ independently for all $u>v$.
For any $C'>0$, there exists some $C>0$ such that
$$\opnorm{T_{\tau}(A)-P}\leq C\sqrt{np_{\max}+1},$$
with probability at least $1-n^{-C'}$ uniformly over $\tau\in [C_1(np_{\max}+1),C_2(np_{\max}+1)]$ for some sufficiently large constants $C_1,C_2$, where $p_{\max}=\max_{u\geq v}P_{uv}$.
\end{lemma}

\begin{lemma}\label{lem:eigen}
For $P = (P_{uv}) = (B_{\sigma(u)\sigma(v)})$, we have SVD $P=U\Lambda U^T$, where
$$U=Z\Delta^{-1}W,$$
with $\Delta=\text{diag}(\sqrt{n_1},...,\sqrt{n_k})$, $Z\in\{0,1\}^{n\times k}$ is a matrix with exactly one nonzero entry  in each row at $(i,\sigma(i))$ taking value $1$ and $W\in O(k,k)$.
\end{lemma}
\begin{proof}
Note that 
$$P=ZBZ^T=Z\Delta^{-1}\Delta B\Delta(Z\Delta^{-1})^T,$$
and observe that $Z\Delta^{-1}\in O(n,k)$. Apply SVD to the matrix $\Delta B\Delta^T=W\Lambda W^T$ for some $W\in O(k,k)$, and then we have $P=U\Lambda U^T$ with $U=Z\Delta^{-1}W\in O(k,k)$.
\end{proof}

\begin{proof}[Proof of Theorem \ref{thm:ini1}]
Under the current assumption, $\mathbb{E}\tau\in[C_1'a,C_2'a]$ for some large $C_1'$ and $C_2'$.
Using Bernstein's inequality, we have $\tau\in[C_1a,C_2a]$ for some large $C_1$ and $C_2$ with probability at least $1-e^{-C'n}$.
When (\ref{eq:assume}) holds, by Lemma \ref{lem:sp1}, we deduce that the $k\Th$ eigenvalue of $T_{\tau}(A)$ is lower bounded by $c_1\lambda_k$ with probability at least $1-n^{-C'}$ for some small constant $c_1\in (0,1)$.
By Davis--Kahan's sin-theta theorem \citep{davis70}, we have $\fnorm{\wh{U}-UW_1}\leq C\frac{\sqrt{k}}{\lambda_k}\opnorm{T_{\tau}(A)-P}$ for some $W_1\in O(k,k)$ and some constant $C>0$. Applying Lemma \ref{lem:eigen}, we have
\begin{equation}
\fnorm{\wh{U}-V}\leq C\frac{\sqrt{k}}{\lambda_k}\opnorm{T_{\tau}(A)-P},\label{eq:useDK}
\end{equation}
where $V=Z\Delta^{-1}W_2\in O(n,k)$ for some $W_2\in O(k,k)$. 
Combining (\ref{eq:useDK}), Lemma \ref{lem:sp1} and the conclusion $\tau\in[C_1a,C_2a]$, we have
\begin{equation}
\fnorm{\wh{U}-V}\leq\frac{C\sqrt{k}\sqrt{a}}{\lambda_k},\label{eq:UVclose}
\end{equation}
with probability at least $1-n^{-C'}$. The definition of $V$ implies that
\begin{equation}
\norm{V_{u*}-V_{v*}} = \sqrt{\frac{1}{n_u}+\frac{1}{n_v}}\indc{\sigma(u)\neq \sigma(v)}. \label{eq:Vchar}
\end{equation}
In other words, define $Q=\Delta^{-1}W_2\in\mathbb{R}^{k\times k}$ and we have $V_{u*}=Q_{\sigma(u)*}$ for each $u\in[n]$.
Hence, for $\sigma(u)\neq \sigma(v)$, $\norm{Q_{\sigma(u)*}-Q_{\sigma(v)*}}=\norm{V_{u*}-V_{v*}}\geq \sqrt{\frac{2k}{\beta n}}$. Recall the definition $r=\mu\sqrt{\frac{k}{n}}$ in Algorithm \ref{algo:greedy}.
Define the sets
$$
T_i=\left\{u\in \sigma^{-1}(i): \norm{\wh{U}_{u*}-Q_{i*}}<\frac{r}{2}\right\},\quad i\in[k].
$$
\begin{figure}[bt]
\centering
\includegraphics[width=0.8\textwidth]{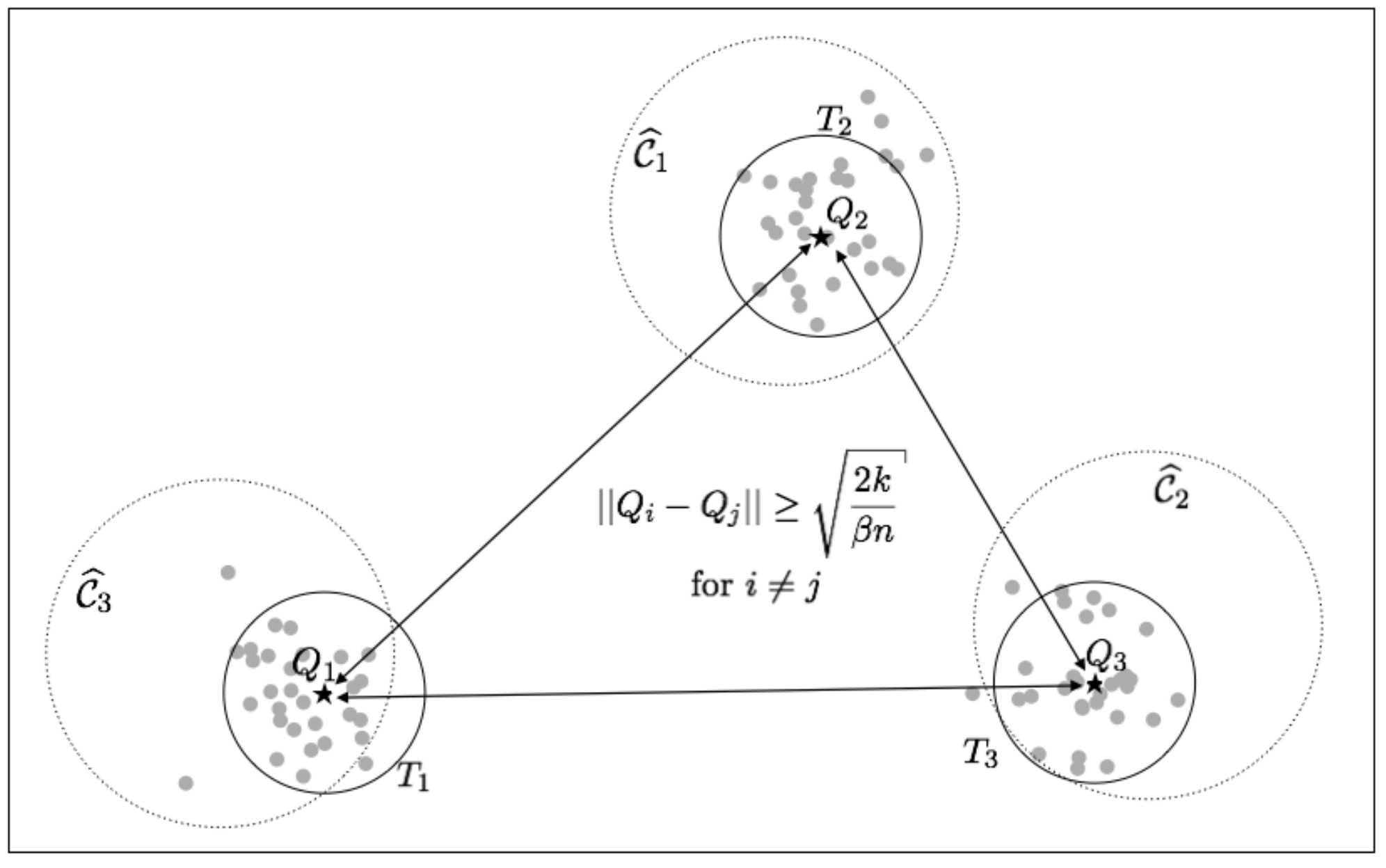}
\caption{\footnotesize{The schematic plot for the proof of Theorem \ref{thm:ini1}. The balls $\{T_i\}_{i\in[k]}$ are centered at $\{Q_i\}_{i\in[k]}$, and the centers are at least $\sqrt{\frac{2k}{\beta n}}$ away from each other. The balls $\{\wh{\mathcal{C}}_i\}_{i\in[k]}$ intersect with large proportions of $\{T_i\}_{i\in[k]}$, and their subscripts do not need to match due to some permutation.}\label{fig:kballs}}
\end{figure}
By definition, $T_i\cap T_j=\varnothing$ when $i\neq j$, and we also have
\begin{equation}
\cup_{i\in[k]}T_i=\left\{u\in[n]: \norm{\wh{U}_{u*}-V_{u*}}<\frac{r}{2}\right\}.\label{eq:defTa}
\end{equation}
Therefore,
$$
\left|\left(\cup_{i\in[k]}T_i\right)^c\right|\frac{r^2}{4} \leq \sum_{u\in[n]}\norm{\wh{U}_{u*}-V_{u*}}^2 \leq \frac{C^2ka}{\lambda_k^2},
$$
where the last inequality is by (\ref{eq:UVclose}). After rearrangement, we have
\begin{equation}
\left|\left(\cup_{i\in[k]}T_i\right)^c\right| \leq \frac{4C^2 na}{\mu^2\lambda_k^2}.\label{eq:set4}
\end{equation}
In other words, most nodes are close to the centers and are in the set (\ref{eq:defTa}). 
Note that the sets $\{T_i\}_{i\in[k]}$ are disjoint. Suppose there is some $i\in[k]$ such that $|T_i|<|\sigma^{-1}(i)|-\left|\left(\cup_{i\in[k]}T_i\right)^c\right|$, we have $\left|\cup_{i\in[k]}T_i\right|=\sum_{i\in[k]}|T_i|<n-\left|\left(\cup_{i\in[k]}T_i\right)^c\right|=\left|\cup_{i\in[k]}T_i\right|$, which is impossible. Thus,
the cardinality of $T_i$ for each $i\in[k]$ is lower bounded as
\begin{equation}
|T_i| \geq |\sigma^{-1}(i)|-\left|\left(\cup_{i\in[k]}T_i\right)^c\right| \geq \frac{n}{\beta k}-\frac{4C^2 na}{\mu^2\lambda_k^2}>\frac{n}{2\beta k},\label{eq:lowerT}
\end{equation}
where the last inequality above is by the assumption (\ref{eq:assume}). Intuitively speaking, except for a negligible proportion, most data points in $\{\wh{U}_{u*}\}_{u\in[n]}$ are very close to the population centers $\{Q_{i*}\}_{i\in[k]}$. Since the centers are at least $\sqrt{\frac{2k}{\beta n}}$ away from each other and $\{T_i\}_{i\in[k]}$ and $\{\wh{C}_i\}_{i\in[k]}$ are both defined through the critical radius $r=\mu\sqrt{\frac{k}{n}}$ for a small $\mu$, each $\wh{C}_i$ should intersect with only one $T_i$ (see Figure \ref{fig:kballs}). We claim that there exists some permutation $\pi$ of the set $[k]$, such that for $\wh{C}_i$ defined in Algorithm \ref{algo:greedy},
\begin{equation}
\wh{\mathcal{C}}_i\cap T_{\pi(i)}\neq \varnothing\quad\text{and}\quad|\wh{\mathcal{C}}_i|\geq|T_{\pi(i)}|\quad\text{for each }i\in[k]. \label{eq:maincon}
\end{equation}
In what follows, we first establish the result of Theorem \ref{thm:ini1} by assuming (\ref{eq:maincon}). The proof of (\ref{eq:maincon}) will be given in the end.
Note that for any $i\neq j$, $T_{\pi(i)}\cap \wh{\mathcal{C}}_j=\varnothing$, which is deduced from the fact that $\wh{\mathcal{C}}_j\cap T_{\pi(j)}\neq\varnothing$ and the definition of $\wh{\mathcal{C}}_j$. Therefore, $T_{\pi(i)}\subset \wh{\mathcal{C}}_j^c$ for all $j\neq i$. Combining with the fact that $T_{\pi(i)}\cap \wh{\mathcal{C}}_i^c\subset \wh{\mathcal{C}}_i^c$, we get $T_{\pi(i)}\cap \wh{\mathcal{C}}_i^c\subset (\cup_{i\in[k]}\wh{\mathcal{C}}_i)^c$. Therefore,
\begin{equation}
\cup_{i\in[k]}\left(T_{\pi(i)}\cap \wh{\mathcal{C}}_i^c\right)\subset \left(\cup_{i\in[k]}\wh{\mathcal{C}}_i\right)^c.\label{eq:set}
\end{equation}
Since $T_i\cap T_j=\varnothing$ for $i\neq j$, we deduce from (\ref{eq:set}) that
\begin{equation}
\sum_{i\in[k]}\left|T_{\pi(i)}\cap \wh{\mathcal{C}}_i^c\right|\leq \left| \left(\cup_{i\in[k]}\wh{\mathcal{C}}_i\right)^c\right|.\label{eq:set2}
\end{equation}
By definition, $\wh{\mathcal{C}}_i\cap \wh{\mathcal{C}}_j=\varnothing$ for $i\neq j$, we deduce from (\ref{eq:maincon}) that
\begin{equation}
\left| \left(\cup_{i\in[k]}\wh{\mathcal{C}}_i\right)^c\right|=n-\sum_{i\in[k]}|\wh{\mathcal{C}}_i|\leq n-\sum_{i\in[k]}|T_i|=\left| \left(\cup_{i\in[k]}T_i\right)^c\right|. \label{eq:set3}
\end{equation}
Combining (\ref{eq:set2}), (\ref{eq:set3}) and (\ref{eq:set4}), we have
\begin{equation}
\sum_{i\in[k]}\left|T_{\pi(i)}\cap \wh{\mathcal{C}}_i^c\right|\leq \frac{4C^2 na}{\mu^2\lambda_k^2}. \label{eq:set5}
\end{equation}
Since for any $u\in\cup_{i\in[k]}(\wh{\mathcal{C}}_i\cap T_{\pi(i)})$, we have $\wh{\sigma}(u)=i$ when $\sigma(u)=\pi(i)$, the mis-classification rate is bounded as
\begin{eqnarray*}
\ell_0(\wh{\sigma},\pi^{-1}(\sigma)) &\leq& \frac{1}{n}\left|\left(\cup_{i\in[k]}(\wh{\mathcal{C}}_i\cap T_{\pi(i)})\right)^c\right| \\
&\leq& \frac{1}{n}\left(\left|\left(\cup_{i\in[k]}(\wh{\mathcal{C}}_i\cap T_{\pi(i)})\right)^c\cap\left(\cup_{i\in[k]}T_i\right)\right| + \left|\left(\cup_{i\in[k]}T_i\right)^c\right|\right) \\
&\leq& \frac{1}{n}\left(\sum_{i\in[k]}\left|T_{\pi(i)}\cap \wh{\mathcal{C}}_i^c\right| + \left|\left(\cup_{i\in[k]}T_i\right)^c\right|\right) \\
&\leq&\frac{8C^2 a}{\mu^2\lambda_k^2},
\end{eqnarray*}
where the last inequality is from (\ref{eq:set5}) and (\ref{eq:set4}). This proves the desired conclusion.

Finally, we are going to establish the claim (\ref{eq:maincon}) to close the proof. We use mathematical induction. For $i=1$, it is clear that $|\wh{\mathcal{C}}_1|\geq\max_{i\in[k]}|T_{i}|$ holds by the definition of $\wh{\mathcal{C}}_1$.
Suppose $\wh{\mathcal{C}}_1\cap T_{i}= \varnothing$ for all $i\in[k]$, and then we must have
$$\left|\left(\cup_{i\in[k]}T_i\right)^c\right|\geq |\wh{\mathcal{C}}_1|\geq\max_{i\in[k]}|T_{i}|\geq \frac{n}{2\beta k},$$
where the last inequality is by (\ref{eq:lowerT}). This contradicts (\ref{eq:set4}) under the assumption (\ref{eq:assume}). Therefore, 
there must be a $\pi(1)$ such that $\wh{\mathcal{C}}_1\cap T_{\pi(1)}\neq \varnothing$ and $|\wh{\mathcal{C}}_1|\geq|T_{\pi(1)}|$. Moreover, 
\begin{eqnarray*}
|\wh{\mathcal{C}}_1^c\cap T_{\pi(1)}| &=& |T_{\pi(1)}|-|T_{\pi(1)}\cap \wh{\mathcal{C}}_1| \\
&\leq&  |\wh{\mathcal{C}}_1|-|T_{\pi(1)}\cap \wh{\mathcal{C}}_1| \\
&=& |\wh{\mathcal{C}}_1\cap T_{\pi(1)}^c| \\
&\leq& \left|\left(\cup_{i\in[k]}T_i\right)^c\right|,
\end{eqnarray*}
where the last inequality is because $T_{\pi(1)}$ is the only set in $\{T_i\}_{i\in[k]}$ that intersects $\wh{\mathcal{C}}_1$ by the definitions. By (\ref{eq:set4}), we get
\begin{equation}
|\wh{\mathcal{C}}_i^c\cap T_{\pi(i)}| \leq \frac{4C^2 na}{\mu^2\lambda_k^2},\label{eq:almost}
\end{equation}
for $i=1$.

Now suppose (\ref{eq:maincon}) and (\ref{eq:almost}) are true for $i=1,...,l-1$. Because of the sizes of $\{\wh{\mathcal{C}}_i\}_{i\in[l-1]}$ and the fact that $\{T_i\}_{i\in[k]}$ are mutually exclusive, we have
$$\left(\cup_{i=1}^{l-1}\wh{\mathcal{C}}_i\right)\cap \left(\cup_{i\in[k]\backslash \cup_{i=1}^{l-1}\{\pi(i)\}}T_{i}\right)=\varnothing.$$
Therefore, for the set $S$ in the current step, $\cup_{i\in[k]\backslash \cup_{i=1}^{l-1}\{\pi(i)\}}T_{i}\subset S$. By the definition of $\wh{\mathcal{C}}_l$, we have $|\wh{\mathcal{C}}_l|\geq \max_{i\in[k]\backslash \cup_{i=1}^{l-1}\{\pi(i)\}}|T_{i}|\geq\frac{n}{2\beta k}$. Suppose $\wh{\mathcal{C}}_l\cap T_{\pi(i)}\neq \varnothing$ for some $i=1,...,l-1$. Then, this $T_{\pi(i)}$ is the only set in $\{T_i\}_{i\in[k]}$ that intersects $\wh{\mathcal{C}}_l$ by their definitions. This implies that
$$|\wh{\mathcal{C}}_l|\leq |\wh{\mathcal{C}}_l\cap T_{\pi(i)}|+\left|\left(\cup_{i\in[k]}T_i\right)^c\right|.$$
Since $\wh{\mathcal{C}}_l\cap \wh{\mathcal{C}}_{\pi(i)}=\varnothing$, $|\wh{\mathcal{C}}_l\cap T_{\pi(i)}|\leq |\wh{\mathcal{C}}_{i}^c\cap T_{\pi(i)}|$ is bounded by (\ref{eq:almost}). Together with (\ref{eq:set4}), we have
$$|\wh{\mathcal{C}}_l|\leq\frac{8C^2na}{\mu^2\lambda_k^2},$$
which contradicts $|\wh{\mathcal{C}}_l|\geq \frac{n}{2\beta k}$ under the assumption (\ref{eq:assume}). Therefore, we must have $\wh{\mathcal{C}}_l\cap T_{\pi(i)}= \varnothing$ for all $i=1,...,l-1$. Now suppose $\wh{\mathcal{C}}_l\cap T_{\pi(i)}= \varnothing$ for all $i\in[k]$, we must have
$$\left|\left(\cup_{i\in[k]}T_i\right)^c\right|\geq |\wh{\mathcal{C}}_l|\geq \frac{n}{2\beta k},$$
which contradicts (\ref{eq:set4}). Hence, $\wh{C}_l\cap T_{\pi(l)}\neq\varnothing$ for some $\pi(l)\in[k]\backslash \cup_{i=1}^{l-1}\{\pi(i)\}$, and (\ref{eq:maincon}) is established for $i=l$. Moreover, (\ref{eq:almost}) can also be established for $i=l$ by the same argument that is used to prove (\ref{eq:almost}) for $i=1$. The proof is complete.
\end{proof}

\subsection{Proof of Theorem \ref{thm:ini2}}\label{sec:xixixi}

Define $P_{\tau}=P+\frac{\tau}{n}\mathbf{1}\mathbf{1}^T$. The proof of the following lemma is given in the appendix.

\begin{lemma} \label{lem:sp2}
Consider a symmetric adjacency matrix $A\in\{0,1\}^{n\times n}$ and a symmetric matrix $P\in[0,1]^{n\times n}$ satisfying $A_{uu}=0$ for all $u\in[n]$ and $A_{uv}\sim \text{Bernoulli}(P_{uv})$ independently for all $u>v$.
For any $C'>0$, there exists some $C>0$ such that
$$\opnorm{L(A_{\tau})-L(P_{\tau})}\leq C\sqrt{\frac{\log(e(np_{\max}+1))}{np_{\max}+1}},$$
with probability at least $1-n^{-C'}$ uniformly over $\tau\in[C_1(np_{\max}+1),C_2(np_{\max}+1)]$ for some sufficiently large constants $C_1,C_2$, where $p_{\max}=\max_{u\geq v}P_{uv}$.
\end{lemma}

\begin{lemma}\label{lem:neweigen}
Consider $P = (P_{uv}) = (B_{\sigma(u)\sigma(v)})$.
Let the SVD of the matrix $L(P_{\tau})$ be $L(P_{\tau})=U\Sigma U^T$, with $U\in O(n,k)$ and $\Sigma=\text{diag}(\sigma_1,...,\sigma_k)$. For $V=UW$ with any $W\in O(r,r)$, we have $\norm{V_{u*}-V_{v*}}=\sqrt{\frac{1}{n_u}+\frac{1}{n_v}}$ when $\sigma(u)\neq \sigma(v)$ and $V_{u*}=V_{v*}$ when $\sigma(u)=\sigma(v)$. Moreover, $\sigma_k\geq\frac{\lambda_k}{2\tau}$ as long as $\tau\geq np_{\max}$.
\end{lemma}
\begin{proof}
The first part is Lemma 1 in \cite{joseph2013impact}. Define $\bar{d}_v=\sum_{u\in[n]}P_{uv}$ and $\bar{D}_{\tau}=\text{diag}(\bar{d}_1+\tau,...,\bar{d}_n+\tau)$. Then, we have $L(P_{\tau})=\bar{D}_{\tau}^{-1/2}P_{\tau}\bar{D}_{\tau}^{-1/2}$. Note that $P_{\tau}$ has an SBM structure so that it has rank at most $k$, and the $k\Th$ eigenvalue of $P_{\tau}$ is lower bounded by $\lambda_k$. Thus, we have
$$\sigma_k\geq\frac{\lambda_k}{\max_{u\in [n]}\bar{d}_u+\tau}.$$
Observe that $\max_{u\in[n]}\bar{d}_u\leq np_{\max}\leq\tau$, and the proof is complete.
\end{proof}

\begin{proof}[Proof of Theorem \ref{thm:ini2}]
As is shown in the proof of Theorem \ref{thm:ini1}, $\tau\in[C_1a,C_2a]$ for some large $C_1,C_2$ with probability at least $1-e^{-C'n}$.
By Davis--Kahan's sin-theta theorem \citep{davis70}, we have $\fnorm{\wh{U}-UW}\leq C_1\frac{\sqrt{k}}{\sigma_k}\opnorm{L(A_{\tau})-L(P_{\tau})}$ for some $W\in O(r,r)$ and some constant $C_1>0$. Let $V=UW$ and apply Lemma \ref{lem:sp2} and Lemma \ref{lem:neweigen}, we have
\begin{equation}
\fnorm{\wh{U}-V}\leq \frac{C\sqrt{k}\sqrt{a\log a}}{\lambda_k}, \label{eq:usenew}
\end{equation}
with probability at least $1-n^{-C'}$. Note that by Lemma \ref{lem:neweigen}, $V$ satisfies (\ref{eq:Vchar}). Replace (\ref{eq:UVclose}) by (\ref{eq:usenew}), and follow the remaining proof of Theorem \ref{thm:ini1}, the proof is complete.
\end{proof}

\begin{small}

\bibliographystyle{plainnat}
\bibliography{network}
\end{small}


\newpage
\appendix

\begin{center}
{\Large Supplement to ``Achieving Optimal Misclassification Proportion in Stochastic Block Model''}\\
~\\
By Chao Gao$^1$, Zongming Ma$^2$, Anderson Y.~Zhang$^1$ and Harrison H.~Zhou$^1$\\
~\\
$^1$Yale University and $^2$University of Pennsylvania
\end{center}

\section{A simplified version of Algorithm \ref{algo:refine}}

\begin{algorithm}[!htb]
	\SetAlgoLined
\caption{A simplified refinement scheme for community detection 
\label{algo:simple}}
\KwIn{
Adjacency matrix $A\in \sth{0,1}^{n\times n}$,\\
~~~~~~~~~~~~number of communities $k$,\\
~~~~~~~~~~~~initial community detection method $\sigma^0$.
}

\KwOut{
Community assignment $\wh\sigma$.
}


\smallskip
{\bf Initialization:}\\
\nl Apply $\sigma^0$ on $A$ to obtain $\sigma^0(u)$ for all $u\in[n]$\;
\nl Define $\wt{\mathcal{C}}_i = \sth{v: \sigma^0(v) = i}$ for all $i\in [k]$;
let $\wt\calE_i$ be the set of edges within $\wt{\mathcal{C}}_i$, and $\wt\calE_{ij}$ the set of edges between $\wt{\mathcal{C}}_i$ and $\wt{\mathcal{C}}_j$ when $i\neq j$\;
\nl Define
$$
\wh{B}_{ii} = \frac{|\wt\calE_i|}{\frac{1}{2}|\wt\calC_i|(|\wt\calC_i|-1) },
\quad 
\wh{B}_{ij} = 
\frac{|\wt\calE_{ij}|}{|\wt\calC_i||\wt\calC_j|},
\quad
\forall i\neq j\in [k],
$$
and let 
$$
\wh{a}= n\min_{i\in [k]} \wh{B}_{ii}
\quad \mbox{and} \quad
\wh{b} = n\max_{i\neq j\in [k]} \wh{B}_{ij}.
$$
{\bf Penalized neighbor voting:}\\
\nl For
$$
t = \frac{1}{2}\log
{\frac{\wh{a}(1 - \wh{b}/n)}{\wh{b}(1 - \wh{a}/n)}},
$$
define
$$
\rho = 
-\frac{1}{2t}\log\pth{\frac{\frac{\wh{a}}{n}e^{-t}+1 -\frac{\wh{a}}{n}}{\frac{\wh{b}}{n}e^{t}+1 -\frac{\wh{b}}{n}}},
$$

\nl For each $u\in[n]$, set 
$$
\wh\sigma(u) = \argmax_{l\in [k]} \sum_{\sigma^0(v) = l} A_{uv} - \rho \sum_{v\in [n]} \indc{\sigma^0(v)= l}.
$$

\end{algorithm}

\section{Proofs of \prettyref{thm:upper-a-ll-b}}
\begin{proof}[Proof of \prettyref{thm:upper-a-ll-b}]
Let us consider $\Theta = \Theta_0(n,k,a,b,\beta)$ and the case of $\Theta(n,k,a,b,\lambda,\beta;\alpha)$ is similar except that the condition \eqref{eq:gamma-cond-2} is needed to establish the counterpart of \prettyref{lmm:refine-nb-2}.
The proof essentially follows the same steps as those in the proof of \prettyref{thm:upper}.
First, we note that \prettyref{lmm:ab-est} continues to hold since it does not need the assumption of $a/b$ being bounded.
Thus, the first job is to establish the counterpart of \prettyref{lmm:refine-nb} with $\eta'$ replaced with $C_\eps \frac{2\epsilon_0}{3}$.
As before, let $p=a/n$ and $q=b/n$.

To this end, we first proceed in the same way to obtain \eqref{eq:pl-decomp} -- \eqref{eq:pl-bound}.
Without loss of generality, let us consider the case where $t^* > \log\frac{2}{\eps_0}$ and $t_u = \log\frac{2}{\eps_0}$ since otherwise we can essentially repeat the proof of \prettyref{thm:upper}.
Note that this implies $\frac{a}{b} > ( \frac{2}{\eps_0})^2$. 
In this case, 
with the new $t_u$ in \eqref{eq:t-set-cap}, we have on the event $E_u$, 
\begin{align*}
(qe^{t_u}+1-q)(pe^{-t_u}+1-p)
= e^{-I'}
\end{align*}
where
\begin{align}
I' & = -\log\pth{\pth{1-p}\pth{1-q} + pq + \pth{e^{t_u-t^*} + e^{t^*-t_u}}\sqrt{\pth{1-p}\pth{1-q} pq}} \nonumber \\
& \geq \pth{1 - C_\eps \frac{3 \epsilon_0}{5}} I^*.
\label{eq:refine-t2-new}
\end{align}
To see this, we first note that for any $x,y\in (0,1)$ and sufficient small constant $c_0>0$, if $y\geq x\geq (1-c_0)y$ and $\frac{y-x}{1-y} \leq 1$, then
\begin{align*}
-\log(1-x) 
& = 
-\log(1-y) - \log\pth{1+\frac{y-x}{1-y}} 
\geq -\log(1-y) - 2\frac{y-x}{1-y} 
\geq -\pth{1 - C'_y c_0 }\log(1-y),
\end{align*}
where $C'_y = \frac{2y}{-(1-y)\log(1-y)}$.
When $\frac{a}{b}>(\frac{2}{\eps_0})^2$ and $t_u = \log\frac{2}{\eps_0}$,
we have $I' = -\log(1-x)$ for 
\begin{align*}
x = p+q-2pq-(e^{t_u-t^*} + e^{t^*-t_u})\sqrt{\pth{1-p}\pth{1-q} pq} 
\geq p-2pq - q e^{t_u} - pe^{-t_u} 
\geq p(1-\eps_0-\frac{\eps_0^2}{2}),
\end{align*}
while $I^* = -\log(1-y)$ for
\begin{align*}
y = p+q-2pq-2\sqrt{\pth{1-p}\pth{1-q} pq}\leq p+q \leq p(1+(\frac{\eps_0}{2})^2).
\end{align*}
Thus, for any $\eps_0 \in (0,c_\eps)$, $1-\frac{\eps}{2} \geq y\geq x\geq (1 - 2\eps_0)y$ and $\frac{y-x}{1-y} \leq 1$, and we apply the inequality in the third last display to obtain \eqref{eq:refine-t2-new}.

Thus, the term in \eqref{eq:fix4} is upper bounded by 
\begin{align*}
\exp\pth{-\big(1-C_\eps \frac{3\eps_0}{5} \big)\frac{n_1+n_l}{2}I^*}.
\end{align*}
On the other hand, since $|e^{-t_u}-1| \leq 1$, $|e^{t_u}-1|$ is bounded and $\frac{p-q}{p}\asymp 1$, the term in \eqref{eq:fix3} continues to be bounded by
\begin{align*}
\exp\pth{-o(1)\frac{n_1+n_l}{2}I^*}.
\end{align*}
Moreover, by the same argument as in \prettyref{lmm:refine-nb}, \eqref{eq:fix200} continues to hold. 
Thus, we can replace \eqref{eq:pl-bound} as
\begin{align*}
p_l\leq \exp\pth{-\big(1-C_\eps \frac{2\eps_0}{3} \big)\frac{n_1+n_l}{2}I^*},
\end{align*}
and so when $k\geq 3$,
\begin{align}
	\label{eq:refine-allb-1}
\Prob\sth{\wh\sigma_u(u) \neq \pi_u(\sigma(u))} \leq 
(k-1)\exp\sth{-\pth{1-C_\eps\frac{2\epsilon_0}{3}}\frac{n I^*}{\beta k}}
+ C n^{-(1+\delta)}
\end{align}
and when $k=2$, we can replace $\beta$ by $1$ in the last display.

When $k\geq 3$,
given the last display and \eqref{eq:consensus-bd}, we have 
\begin{align}
	\Prob\sth{\wh\sigma(u) \neq \sigma(u)} 
	& = \Prob\sth{\xi_u(\wh\sigma_u(u))\neq \sigma(u)} \nonumber \\
	& \leq \Prob\sth{\xi_u(\wh\sigma_u(u))\neq \sigma(u),\, \xi_u = \pi_u^{-1}} 
	+ \Prob\sth{\xi_u \neq \pi_u^{-1}} \nonumber \\
	& \leq \Prob\sth{\wh\sigma_u(u)\neq \pi_u(\sigma(u))} + \Prob\sth{\xi_u \neq \pi_u^{-1}} \nonumber \\
	& \leq Cn^{-(1+\delta)} + (k-1) \exp\sth{-\pth{1-C_\eps\frac{2\epsilon_0}{3}}\frac{nI^*}{\beta k}}  .
	\label{eq:refine-allb-2}
\end{align}
Thus, the assumption that $\frac{(a-b)^2}{ak\log k}\rightarrow\infty$ and Markov's inequality leads to
\begin{align}
& 
\hskip -4em \Prob\sth{\loss_0(\sigma, \wh\sigma) > 
\exp\sth{-(1-C_\eps\epsilon_0)\frac{nI^*}{\beta k}}
} 
\nonumber 
\\
& \leq \Prob\sth{\loss_0(\sigma, \wh\sigma) > 
(k-1) \exp\sth{-(1-C_\eps\frac{5\epsilon_0}{6})\frac{nI^*}{\beta k}}
} 
\nonumber 
\\
& \leq \frac{1}{ (k-1) \exp\sth{-(1-C_\eps\frac{5\epsilon_0}{6})\frac{nI^*}{\beta k}} } \frac{1}{n}\sum_{u=1}^n \Prob\sth{\wh\sigma(u) \neq \sigma(u)} 
\nonumber
\\
& \leq \exp\sth{- \frac{C_\eps \epsilon_0}{6} \frac{nI^*}{\beta k}} 
+ \frac{Cn^{-(1+\delta)}}{(k-1) \exp\sth{-(1-C_\eps\frac{5\epsilon_0}{6})\frac{nI^*}{\beta k}} }.
\label{eq:refine-allb-3}
\end{align}
If $(k-1) \exp\sth{-(1- C_\eps\frac{5\epsilon_0}{6})\frac{nI^*}{\beta k}} \geq n^{-(1+\delta/2)}$, then
\begin{align}
\Prob\sth{\loss_0(\sigma, \wh\sigma) > 
\exp\sth{-(1-C_\eps \epsilon_0)\frac{nI^*}{\beta k}}
} 
\leq \exp\sth{-\frac{C_\eps \epsilon_0}{6}{n I^* \over \beta k}} + Cn^{-\delta/2} = o(1).
\label{eq:refine-allb-4}
\end{align}
If $(k-1) \exp\sth{-(1-C_\eps\frac{5\epsilon_0}{6})\frac{nI^*}{\beta k}} < n^{-(1+\delta/2)}$, then 
\begin{align}
\Prob\sth{\loss_0(\sigma, \wh\sigma) > 
\exp\sth{-(1-C_\eps \epsilon_0)\frac{nI^*}{\beta k}}
} \leq
\Prob\sth{\loss_0(\sigma, \wh\sigma) > 0}
\leq \sum_{u=1}^n \Prob\sth{\wh\sigma(u)\neq \sigma(u)} ~~~ & 
\nonumber
\\
\leq n(k-1) \exp\sth{-(1-C_\eps \frac{2\epsilon_0}{3})\frac{nI^*}{\beta k}} + Cn^{-\delta} 
\leq Cn^{-\delta/2} = o(1). & 
\label{eq:refine-allb-5}
\end{align}
Here, the second last inequality holds since $ (k-1)\exp\sth{-(1-C_\eps \frac{2\epsilon_0}{3})\frac{nI^*}{\beta k}} < \exp\sth{-(1-C_\eps\frac{5\epsilon_0}{6})\frac{nI^*}{\beta k} } < n^{-(1+\delta/2)}$.
We complete the proof for the case of 
$k\geq  3$ by noting that no constant or sequence in the foregoing arguments involves $B,\sigma$ or $u$.
When 
$k=2$, we run the foregoing arguments with $\beta$ replaced by $1$ to obtain the desired claim.
\end{proof}

\section{Proofs of Theorems \ref{cor:upper1}, \ref{cor:upper2} and \ref{thm:upper-bound-ab-known}}

\begin{prop}\label{prop:eigen}
For SBM in the space $\Theta_0(n,k,a,b,\beta)$ satisfying $n\geq 2\beta k$, we have $\lambda_k\geq \frac{a-b}{\beta k}$.
\end{prop}
\begin{proof}
Since the eigenvalues of $P$ are invariant with respect to permutation of the community labels, we consider the case where $\sigma(u)=i$ for $u\in\left\{\sum_{j=1}^{i-1}n_j-1,\sum_{j=1}^in_j\right\}$ without loss of generality, where $\sum_{j=1}^0n_j=0$. Let us use the notation $\mathbf{1}_{d}\in\mathbb{R}^{d}$ and $\mathbf{0}_{d}\in\mathbb{R}^{d}$ to denote the vectors with all entries being $1$ and $0$ respectively. Then, it is easy to check that
$$P-\frac{b}{n}\mathbf{1}_{n}\mathbf{1}_{n}^T=\frac{a-b}{n}\sum_{i=1}^kv_iv_i^T,$$
where $v_1=(\mathbf{1}_{n_1}^T,\mathbf{0}_{n_2}^T,...,\mathbf{0}_{n_k}^T)^T$, $v_1=(\mathbf{0}_{n_1}^T,\mathbf{1}_{n_2}^T,\mathbf{0}_{n_3}^T,...,\mathbf{0}_{n_k}^T)^T$,..., $v_k=(\mathbf{0}_{n_1}^T,...,\mathbf{0}_{n_{k-1}}^T,\mathbf{1}_{n_k}^T)^T$. Note that $\{v_i\}_{i=1}^k$ are orthogonal to each other, and therefore
$$\lambda_k\left(\sum_{i=1}^kv_iv_i^T\right)\geq \min_{i\in[k]}n_i\geq \frac{n}{\beta k}-1\geq\frac{n}{2\beta k}.$$
By Weyl's inequality (Theorem 4.3.1 of \cite{horn2012matrix}), 
$$\lambda_k(P)\geq \frac{a-b}{n}\lambda_k\left(\sum_{i=1}^kv_iv_i^T\right)+\lambda_n\left(\frac{b}{n}\mathbf{1}_{n}\mathbf{1}_{n}^T\right)\geq \frac{a-b}{2\beta k}.$$
This completes the proof.
\end{proof}

\begin{proof}[Proof of Theorem \ref{cor:upper1}]
Let us first consider $\Theta_0(n,k,a,b,\beta)$.
By Theorem \ref{thm:ini1} and Proposition \ref{prop:eigen}, the misclassification proportion is bounded by $C\frac{k^2a}{(a-b)^2}$ under the condition $\frac{k^3a}{(a-b)^2}\leq c$ for some small $c$. Thus, Condition \ref{cond:init} holds when $\frac{k^3a}{(a-b)^2}=o(1)$, which leads to the desired conclusion in view of Theorem \ref{thm:upper} and Theorem \ref{thm:upper-a-ll-b}. The proof of the space $\Theta(n,k,a,b,\lambda,\beta;\alpha)$ follows the same argument.
\end{proof}

\begin{proof}[Proof of Theorem \ref{cor:upper2}]
The proof is the same as that of Theorem \ref{cor:upper1}.
\end{proof}

\begin{proof}[Proof of Theorem \ref{thm:upper-bound-ab-known}]
When the parameters $a$ and $b$ are known, we can use $\tau=Ca$ for some sufficiently large $C>0$ for both USC$(\tau)$ and NSC$(\tau)$. Then, the results of Theorem \ref{thm:ini1} and Theorem \ref{thm:ini2} hold without assuming $a\leq C_1b$ or fixed $k$. Moreover, $\wh{a}_u$ and $\wh{b}_u$ in (\ref{eq:t-set}) and (\ref{eq:t-set-cap}) can be replaced by $a$ and $b$. Then, the conditions (\ref{eq:gamma-cond-1}) and (\ref{eq:gamma-cond-2}) in Theorem \ref{thm:upper} and Theorem \ref{thm:upper-a-ll-b} can be weakened as $\gamma=o(k^{-1})$ because the we do not need to establish Lemma \ref{lmm:ab-est} anymore. Combining Theorem \ref{thm:upper}, Theorem \ref{thm:ini1}, Theorem \ref{thm:ini2} and Theorem \ref{thm:upper-a-ll-b}, we obtain the desired results.
\end{proof}

\section{Proofs of Lemma \ref{lem:sp1} and Lemma \ref{lem:sp2}}

The following lemma is Corollary A.1.10 in \cite{alon04}.

\begin{lemma}\label{lem:alon}
For independent Bernoulli random variables $X_u\sim\text{Bern}(p_u)$ and $p=\frac{1}{n}\sum_{u\in[n]}p_u$, we have
$$\mathbb{P}\left(\sum_{u\in[n]}(X_u-p_u)\geq t\right)\leq\exp\left(t-(pn+t)\log\left(1+\frac{t}{pn}\right)\right),$$
for any $t\geq 0$.
\end{lemma}

The following result is Lemma 3.5 in \cite{chin15}.

\begin{lemma} \label{lem:vu}
Consider any adjacency matrix $A\in\{0,1\}^{n\times n}$ for an undirected graph. Suppose $\max_{u\in[n]}\sum_{v\in[n]}A_{uv}\leq \gamma$ and for any $S,T\subset[n]$, one of the following statements holds with some constant $C>0$:
\begin{enumerate}
\item $\frac{e(S,T)}{|S||T|\frac{\gamma}{n}}\leq C$,
\item $e(S,T)\log\left(\frac{e(S,T)}{|S||T|\frac{\gamma}{n}}\right)\leq C|T|\log\frac{n}{|T|}$,
\end{enumerate}
where $e(S,T)$ is the number of edges connecting $S$ and $T$.
Then, $\sum_{(u,v)\in H}x_uA_{uv}y_v\leq C'\sqrt{\gamma}$ uniformly over all unit vectors $x,y$, where $H=\{(u,v): |x_uy_v|\geq \sqrt{\gamma}/n\}$ and $C'>0$ is some constant.
\end{lemma}

The following lemma is critical for proving both theorems.

\begin{lemma} \label{lem:remain}
For any $\tau>C(1+np_{\max})$ with some sufficiently large $C>0$, we have
$$|\{u\in[n]: d_u\geq \tau\}|\leq \frac{n}{\tau}$$
with probability at least $1-e^{-C'n}$ for some constant $C'>0$.
\end{lemma}
\begin{proof}
Let us consider any fixed subset of nodes $S\subset[n]$ such that it has degree at least $\tau$ and $|S|=l$ for some $l\in[n]$. Let $e(S)$ be the number of edges in the subgraph $S$ and $e(S,S^c)$ be the number of edges connecting $S$ and $S^c$. By the requirement on $S$,
 either $e(S)\geq C_1l\tau$ or $e(S,S^c)\geq C_1l\tau$ for some universal constant $C_1>0$. We are going to show that both $\mathbb{P}\left(e(S)\geq C_1l\tau\right)$ and $\mathbb{P}\left(e(S,S^c)\geq C_1l\tau\right)$ are small. Note that $\mathbb{E}e(S)\leq C_2l^2p_{\max}$ and $\mathbb{E}e(S,S^c)\leq C_2lnp_{\max}$ for some universal $C_2>0$. Then, when $\tau>C(np_{\max}+1)$ for some sufficiently large $C>0$, Lemma \ref{lem:alon} implies
$$\mathbb{P}\left(e(S)\geq C_1l\tau\right)\leq \exp\left(-\frac{1}{4}C_1l\tau\log\left(1+\frac{C_1\tau}{2C_2l p_{\max}}\right)\right),$$
and
$$\mathbb{P}\left(e(S,S^c)\geq C_1l\tau\right)\leq\exp\left(-\frac{1}{4}C_1l\tau\log\left(1+\frac{C_1\tau}{2C_2n p_{\max}}\right)\right).$$
Applying union bound, the probability that the number of nodes with degree at least $\tau$ is greater than $\xi n$ is 
\begin{eqnarray*}
&& \mathbb{P}\Big(|\{u\in[n]: d_u\geq \tau\}|>\xi n\Big)\\
&\leq& \sum_{l>\xi n} \mathbb{P}\Big(|\{u\in[n]: d_u\geq \tau\}|=l\Big) \\
&\leq& \sum_{l>\xi n}\sum_{|S|=l}\left(\mathbb{P}\left(e(S)\geq C_1l\tau\right)+\mathbb{P}\left(e(S,S^c)\geq C_1l\tau\right)\right) \\
&\leq& \sum_{l>\xi n}\exp\left(l\log\frac{en}{l}\right)\left(\exp\left(-\frac{1}{4}C_1l\tau\log\left(1+\frac{C_1\tau}{2C_2l p_{\max}}\right)\right)\right.\\
&& \left.+\exp\left(-\frac{1}{4}C_1l\tau\log\left(1+\frac{C_1\tau}{2C_2n p_{\max}}\right)\right)\right)\\
&\leq& \sum_{l>\xi n}2\exp\left(l\log\frac{en}{l}-\frac{1}{4}C_1l\tau\log\left(1+\frac{C_1\tau}{2C_2n p_{\max}}\right)\right) \\
&\leq& \exp(-C'n),
\end{eqnarray*} 
where the last inequality is by choosing $\xi=\tau^{-1}$. Therefore, with probability at least $1-e^{-C'n}$, the number of nodes with degree at least $\tau$ is bounded by $\tau^{-1}n$.
\end{proof}

\begin{lemma} \label{lem:core}
Given $\tau>0$,
define the subset $J=\{u\in[n]: d_u\leq\tau\}$. Then for any $C'>0$, there is some $C>0$ such that
$$\opnorm{A_{JJ}-P_{JJ}}\leq C\left(\sqrt{np_{\max}}+\sqrt{\tau}+\frac{np_{\max}}{\sqrt{\tau}+\sqrt{np_{\max}}}\right),$$
with probability at least $1-n^{-C'}$.
\end{lemma}
\begin{proof}
The idea of the proof follows the argument in \cite{friedman1989second,feige2005spectral}.
By definition,
$$\opnorm{A_{JJ}-P_{JJ}}=\sup_{x,y\in S^{n-1}}\sum_{(u,v)\in J\times J} x_u(A_{uv}-P_{uv})y_v.$$
Define $L=\{(u,v): |x_uy_v|\leq (\sqrt{\tau}+\sqrt{p_{\max}n})/n\}$ and $H=\{(u,v): |x_uy_v|\geq (\sqrt{\tau}+\sqrt{p_{\max}n})/n\}$, then we have
$$\opnorm{A_{JJ}-P_{JJ}}\leq \sup_{x,y\in S^{n-1}}\sum_{(u,v)\in L\cap J\times J}x_u(A_{uv}-P_{uv})y_v+\sup_{x,y\in S^{n-1}}\sum_{(u,v)\in H\cap J\times J}x_u(A_{uv}-P_{uv})y_v.$$
A discretization argument in \cite{chin15} implies that
\begin{eqnarray*}
\sup_{x,y\in S^{n-1}}\sum_{(u,v)\in L\cap J\times J}x_u(A_{uv}-P_{uv})y_v &\lesssim& \max_{x,y\in\mathcal{N}}\max_{S\subset[n]}\sum_{(u,v)\in L\cap S\times S}x_u(A_{uv}-\mathbb{E} A_{uv})y_v \\
&& + \max_{x,y\in\mathcal{N}}\max_{S\subset[n]}\sum_{(u,v)\in L\cap S\times S}x_u(\mathbb{E}A_{uv}-P_{uv})y_v,
\end{eqnarray*}
where $\mathcal{N}\subset S^{n-1}$ and $|\mathcal{N}|\leq 5^n$. Then, Bernstein's inequality and union bound imply that $\max_{x,y\in\mathcal{N}}\max_{S\subset[n]}\sum_{(u,v)\in L\cap S\times S}x_u(A_{uv}-\mathbb{E}A_{uv})y_v\leq C(\sqrt{\tau}+\sqrt{n p_{\max}})$ with probability at least $1-e^{-C'n}$. We also have $\max_{x,y\in\mathcal{N}}\max_{S\subset[n]}\sum_{(u,v)\in L\cap S\times S}x_u(\mathbb{E}A_{uv}-P_{uv})y_v\leq\opnorm{\mathbb{E}A-P}\leq 1$.
 This completes the first part.

To bound the second part $\sup_{x,y\in S^{n-1}}\sum_{(u,v)\in H\cap J\times J}x_u(A_{uv}-P_{uv})y_v$, we are going to bound $\sup_{x,y\in S^{n-1}}\sum_{(u,v)\in H\cap J\times J}x_uA_{uv}y_v$ and $\sup_{x,y\in S^{n-1}}\sum_{(u,v)\in H\cap J\times J}x_uP_{uv}y_v$ separately. By the definition of $H$,
$$\sup_{x,y\in S^{n-1}}\sum_{(u,v)\in H\cap J\times J}x_uP_{uv}y_v=\sup_{x,y\in S^{n-1}}\sum_{(u,v)\in H\cap J\times J}\frac{x_u^2y_v^2}{|x_uy_v|}P_{uv}\leq \frac{np_{\max}}{\sqrt{\tau}+\sqrt{p_{\max}n}}.$$
To bound $\sup_{x,y\in S^{n-1}}\sum_{(u,v)\in H\cap J\times J}x_uA_{uv}y_v$, it is sufficient to check the conditions of Lemma \ref{lem:vu} for the graph $A_{JJ}$. By definition, its degree is bounded by $\tau$. Following the argument of \cite{lei14}, the two conditions of Lemma \ref{lem:vu} hold with $\gamma=\tau+np_{\max}$ with probability at least $1-n^{-C'}$. Thus, $\sup_{x,y\in S^{n-1}}\sum_{(u,v)\in H\cap J\times J}x_uA_{uv}y_v\leq C(\sqrt{\tau}+\sqrt{np_{\max}})$ with probability at least $1-n^{-C'}$. Hence, the proof is complete.
\end{proof}

\begin{proof}[Proof of Lemma \ref{lem:sp1}]
By triangle inequality,
$$\opnorm{T_{\tau}(A)-P}\leq \opnorm{T_{\tau}(A)-T_{\tau}(P)}+\opnorm{T_{\tau}(P)-P},$$
where $T_{\tau}(P)$ is the matrix obtained by zeroing out the $u\Th$ row and column of $P$ with $d_u\geq\tau$.
Let $J=\{u\in[n]: d_u\leq\tau\}$, and then $\opnorm{T_{\tau}(A)-T_{\tau}(P)}=\opnorm{A_{JJ}-P_{JJ}}$, whose bound has been established in Lemma \ref{lem:core}. By Lemma \ref{lem:remain}, $|J^c|\leq n/\tau$ with high probability. This implies $\opnorm{T_{\tau}(P)-P}\leq \fnorm{T_{\tau}(P)-P}\leq\sqrt{2n|J^c|p_{\max}^2}\leq \frac{\sqrt{2}np_{\max}}{\sqrt{\tau}}$. Taking $\tau\in[C_1(1+np_{\max}),C_2(1+np_{\max})]$, the proof is complete.
\end{proof}

Now let us prove Lemma \ref{lem:sp2}. The following lemma, which controls the degree, is Lemma 7.1 in \cite{le2015sparse}.

\begin{lemma} \label{lem:degreeset}
For any $C'>0$, there exists some $C>0$ such that with probability at least $1-n^{-C'}$, there exists a subset $J\subset [n]$ satisfying $n-|J|\leq \frac{n}{2e(np_{\max}+1)}$ and
$$|d_v-\mathbb{E}d_v|\leq C\sqrt{(np_{\max}+1)\log(e(np_{\max}+1))},\quad\text{for all }v\in J,$$
where $d_v=\sum_{u\in[n]}A_{uv}$.
\end{lemma}

Using this lemma, together with Lemma \ref{lem:remain} and Lemma \ref{lem:core}, we are able to prove the following result, which improves the bound in Theorem 7.2 of \cite{le2015sparse}.

\begin{lemma} \label{lem:Lcore}
For any $C'>0$, there exists some $C>0$ such that with probability at least $1-n^{-C'}$, there exists a subset $J\subset[n]$ satisfying $n-|J|\leq n/d$ and
$$\opnorm{(L(A_{\tau})-L(P_{\tau}))_{J\times J}}\leq C\left(\frac{\sqrt{d\log d}(d+\tau)}{\tau^2}+\frac{\sqrt{d}}{\tau}\right),$$
where $d=e(np_{\max}+1)$.
\end{lemma}
\begin{proof}
Let us use the notation $d_v=\sum_{u\in[n]}A_{uv}$ in the proof. Define the set $J_1=\left\{v\in[n]: d_v\leq C_1d\right\}$ for some sufficiently large constant $C_1>0$. Using Lemma \ref{lem:remain} and Lemma \ref{lem:core}, with probability at least $1-n^{-C'}$, we have
\begin{equation}
n-|J_1|\leq \frac{n}{2d},\label{eq:J11}
\end{equation}
and
\begin{equation}
\opnorm{(A-P)_{J_1J_1}}\leq C\sqrt{d}.\label{eq:J12}
\end{equation}
Let $J_2$ be the subset in Lemma \ref{lem:degreeset}, and then with probability at least $1-n^{-C'}$, $J_2$ satisfies
\begin{equation}
n-|J_2|\leq \frac{n}{2d},\label{eq:J21}
\end{equation}
and
\begin{equation}
|d_v-\mathbb{E}d_v|\leq C\sqrt{d\log d},\quad\text{for all }v\in J_2.\label{eq:J22}
\end{equation}
Define $J=J_1\cap J_2$. By (\ref{eq:J11}) and (\ref{eq:J21}), we have 
\begin{equation}
n-|J| = |(J_1\cap J_2)^c| \leq |J_1^c| + |J_2^c| = n-|J_1| + n-|J_2| \leq \frac{n}{d},
\end{equation}
and
\begin{equation}
\opnorm{(A-P)_{JJ}}\leq \opnorm{(A-P)_{J_1J_1}}\leq C\sqrt{d}. \label{eq:JAP}
\end{equation}
Moreover, (\ref{eq:J22}) implies
$$
\max_{v\in J}|d_v-\mathbb{E}d_v|\leq C\sqrt{d\log d}.
$$
Define $\bar{d}_v=\sum_{u\in[n]}P_{uv}$. Then,
\begin{equation}
\max_{v\in J}|d_v-\bar{d}_v|\leq \max_{v\in J}|d_v-\mathbb{E}d_v|+1\leq C\sqrt{d\log d}.\label{eq:Jdegree}
\end{equation}
Define $D_{\tau}=\text{diag}(d_1+\tau,...,d_n+\tau)$ and $\bar{D}_{\tau}=\text{diag}(\bar{d}_1+\tau,...,\bar{d}_n+\tau)$. We introduce the notation
$$R=(A_{\tau})_{JJ},\quad B=(D_{\tau})_{JJ}^{-1/2},\quad \bar{R}=(P_{\tau})_{JJ},\quad\bar{B}=(\bar{D}_{\tau})_{JJ}^{-1/2}.$$
Using (\ref{eq:Jdegree}), we have
$$\opnorm{B-\bar{B}}\leq \max_{v\in[n]}\left|\frac{1}{\sqrt{d_v+\tau}}-\frac{1}{\sqrt{\bar{d}_v+\tau}}\right|\leq C\frac{\sqrt{d\log d}}{\tau^{3/2}},$$
for some constant $C>0$. The definitions of $B$ and $\bar{B}$ implies $\opnorm{B}\vee\opnorm{\bar{B}}\leq\frac{1}{\sqrt{\tau}}$. We rewrite the bound (\ref{eq:JAP}) as $\opnorm{R-\bar{R}}\leq C\sqrt{d}$. Since all entries of $\mathbb{E}{A}_{\tau}$ is bounded by $(\tau+d)/n$, we have $\opnorm{\bar{R}}\leq \opnorm{\mathbb{E}{A}_{\tau}}\leq d+\tau$. Therefore, $\opnorm{R}\leq \opnorm{\bar{R}}+\opnorm{R-\bar{R}}\leq C(d+\tau)$. Finally,
\begin{eqnarray*}
&& \opnorm{(L(A_{\tau})-L(P_{\tau}))_{J\times J}} \\
&\leq& \opnorm{B}\opnorm{R}\opnorm{B-\bar{B}} + \opnorm{B}\opnorm{R-\bar{R}}\opnorm{\bar{B}} + \opnorm{B-\bar{B}}\opnorm{\bar{R}}\opnorm{\bar{B}} \\
&\leq& C\left(\frac{\sqrt{d\log d}(d+\tau)}{\tau^2}+\frac{\sqrt{d}}{\tau}\right).
\end{eqnarray*}
The proof is complete.
\end{proof}

\begin{proof}[Proof of Lemma \ref{lem:sp2}]
Recall that $d=np_{\max}+1$.
Following the proof of Theorem 8.4 in \cite{le2015sparse}, it can be shown that with probability at least $1-n^{-C'}$,  for any $J\subset[n]$ such that $n-|J|\leq n/d$,
$$\opnorm{L(A_{\tau})-L(P_{\tau})}\leq \opnorm{(L(A_{\tau})-L(P_{\tau}))_{JJ}}+C\left(\frac{1}{\sqrt{d}}+\sqrt{\frac{\log d}{\tau}}\right),$$
where the first term on the right side of the inequality above is bounded in Lemma \ref{lem:Lcore} by choosing an appropriate $J$. Hence, with probability at least $1-2n^{-C'}$,
$$\opnorm{L(A_{\tau})-L(P_{\tau})}\leq C\left(\frac{\sqrt{d\log d}(d+\tau)}{\tau^2}+\frac{\sqrt{d}}{\tau}\right)+C\left(\frac{1}{\sqrt{d}}+\sqrt{\frac{\log d}{\tau}}\right).$$
Choosing $\tau\in[C_1(1+np_{\max}),C_2(1+np_{\max})]$, the proof is complete.
\end{proof}

\end{document}